\documentclass[10pt,reqno]{amsart}
\usepackage{amsmath} 
\usepackage{amssymb}
\usepackage{dsfont}
\usepackage[dvips,draft,final]{graphics}
\usepackage[T1]{fontenc}
\usepackage{lmodern}
\usepackage{fancyhdr}
\usepackage{url}
\usepackage{enumerate}
\usepackage{color}
\usepackage{graphicx}
\usepackage{mathrsfs}

\newtheorem{theorem}{Theorem}[section]
\newtheorem{proposition}[theorem]{Proposition}
\newtheorem{lemma}[theorem]{Lemma}
\newtheorem{follow}[theorem]{Corollary}

\theoremstyle{definition}
\newtheorem{remark}[theorem]{Remark}

\newcommand{\bel}{\begin{equation} \label}
\newcommand{\ee}{\end{equation}}

\newcommand{\pd}{\partial}

\newcommand{\C}{{\mathbb C}}
\newcommand{\R}{{\mathbb R}}
\newcommand{\N}{{\mathbb N}}

\newcommand{\re}{\mathfrak R}
\newcommand{\im}{\mathfrak I}

\newcommand{\B}{{\mathcal B}}
\newcommand{\cC}{{\mathcal C}}

\newcommand{\cH}{{\mathcal H}}
\newcommand{\sH}{{\mathscr H}}

\newcommand{\U}{{\mathcal U}}

\newcommand{\Z}{{\mathbb Z}}
\def\epsilon{\varepsilon}
\def\phi {\varphi}

\def\beq{\begin{equation}}
\def\eeq{\end{equation}}
\renewcommand{\leq}{\leqslant}
\renewcommand{\geq}{\geqslant}
\newcommand{\bea}{\begin{eqnarray}}
\newcommand{\eea}{\end{eqnarray}}
\newcommand{\beas}{\begin{eqnarray*}}
\newcommand{\eeas}{\end{eqnarray*}}

{

\providecommand{\abs}[1]{\left\lvert#1\right\rvert}
\providecommand{\norm}[1]{\left\lVert#1\right\rVert}

\title[Stability result for elliptic inverse periodic coefficient problem]{Stability result for elliptic inverse periodic coefficient problem by partial Dirichlet-to-Neumann map}
\author{Mourad Choulli, Yavar Kian, Eric Soccorsi}

\begin{document}
\begin{abstract}
We study the inverse problem of identifying a periodic potential perturbation of the Dirichlet Laplacian acting in an infinite cylindrical domain, whose cross section is assumed to be bounded. We prove log-log stable determination of the potential with respect to the partial Dirichlet-to-Neumann map, where the Neumann data is taken on slightly more than half of the boundary of the domain.
\end{abstract}
\maketitle


\section{Introduction}
\label{sec-intro}
\setcounter{equation}{0}
Let $\Omega := \R \times \omega$, where 
$\omega$ is a bounded domain of $\mathbb{R}^2$ which contains the origin, with $C^2$-boundary. Throughout the entire text we denote the generic point $x \in \Omega$ by $x=(x_1,x')$, where $x_1 \in \R$ and $x':= (x_2,x_3) \in \omega$. Given $V \in L^\infty(\Omega)$, real-valued and $1$-periodic w.r.t. $x_1$, i.e.
\bel{eq-per}
V(x_1+1,x')=V(x_1,x'),\ x' \in \omega,\ x_1 \in \R,
\ee
we consider the following boundary value problem (abbreviated as BVP):
\bel{eq1}
\left\{ 
\begin{array}{rcll} 
(-\Delta + V) v & = & 0, & \mbox{in}\ \Omega,\\ 
v & = & f ,& \mbox{on}\ \Gamma : = \pd \Omega.
\end{array}
\right.
\ee
Since $\Gamma= \R \times \pd \omega$, the outward unit vector $\nu$ normal to $\Gamma$ reads
$$ \nu(x_1,x')=(0,\nu'(x')),\ x=(x_1,x')\in\Gamma, $$
where $\nu'$ is the outer unit normal vector of $\pd \omega$.
Therefore, for notational simplicity, we shall refer to $\nu$ for both exterior unit vectors normal to $\pd  \omega$ and to $\Gamma$.
Next for $\xi \in \mathbb{S}^1 :=\{ y\in\R^2;\ \abs{y}=1\}$ fixed, we introduce the $\xi$-illuminated (resp., $\xi$-shadowed) face of $\pd \omega$, as
\bel{xi-isf} 
\pd \omega_{\xi}^- := \{ x \in \pd \omega;\ \xi \cdot \nu(x) \leq 0 \}\ (\mbox{resp.},\ \pd \omega_{\xi}^+= \{x \in \pd \omega;\ \xi \cdot \nu(x) \geq 0\}).
\ee
Here and in the remaining part of this text, we denote by 
$x \cdot y := \sum_{j=1}^k x_j y_j$
the Euclidian scalar product of any two vectors $x:=(x_1,\ldots,x_k)$ and $y:=(y_1,\ldots,y_k)$ of $\R^k$, for $k \in \N^*$, and we put $|x|:=(x \cdot x)^{1 \slash 2}$.

Set $G:=\R\times G'$, where $G'$ is an arbitrary closed neighborhood of $\pd \omega_{\xi}^-$ in $\pd \omega$. In the present paper we seek stability in the determination of $V$ from the knowledge of the partial Dirichlet-to-Neumann (DN) map
\bel{a1}
\Lambda_V :  f  \mapsto {\pd_\nu v}_{|G},
\ee
where $\pd_\nu v (x) := \nabla v (x) \cdot \nu(x)$ is the normal derivative of the solution $v$ to \eqref{eq1}, computed at $x \in \Gamma$.
Otherwise stated we aim for recovering the $1$-periodic electric perturbation $V$ of the Dirichlet Laplacian in the waveguide $\Omega$, by probing the system
with voltage $f$ at the boundary and measuring the current $\pd_\nu u$ on the sub-part $G$ of $\Gamma$.
From a physics viewpoint, this amounts to estimating the impurity potential perturbing the guided propagation in periodic media such as crystals.
 
\subsection{A short bibliography}

Inverse coefficient problems in elliptic partial differential equations such as the celebrated Calder\'on problem have attracted many attention in recent years.
In \cite{SU}, one of the first mathematical papers dealing with this problem, Sylvester and Uhlmann showed in dimension $n \geq 3$ that the full DN map (both the input and the output are taken on the whole boundary of the domain) uniquely determines a smooth conductivity coefficient. The case of $C^1$ conductivities or Lipschitz conductivities sufficiently close to the identity, is treated by \cite{HT}. The identifiability of an unknown coefficient from partial knowledge of the DN map was first proved in \cite{BU}. Assuming that the voltage (i.e. the Dirichlet data) is prescribed everywhere, Bukhgeim and Uhlmann claimed unique determination of the conductivity even when the current measurement (i.e. the Neumann data) is taken on slightly more than half of the boundary. Kenig, Sj\"ostrand and Uhlmann improved this result in \cite{KSU} by
taking both the Dirichlet and the Neumann observations on a neighborhood of, respectively, the back and the front face illuminated by a point light-source lying outside of the convex hull of the domain. Identification for the corresponding two-dimensional inverse problem was treated by Bukhgeim in \cite{B} with the full data, and by Imanuvilov, Uhlmann and Yamamoto in \cite{IUY1, IUY2} with partial data. 

The stability issue for $n \geq 3$ was first addressed in \cite{Al} by Alessandrini, who established a log-type stability estimate for the conductivity from the full DN map. Later on, in \cite{HW}, Heck and Wang proved a log-log-type stability estimate with respect to the partial DN map of \cite{BU}. More recently, in \cite{T}, Tzou showed
that both the magnetic field and the electric potential depend stably on the DN map, even when the Neumann boundary measurement is taken only on a subset that is slightly larger than half of the boundary. In \cite{CDR1, CDR2}, Caro, Dos Santos Ferreira and Ruiz derived a log-log stability estimate for the electric potential from the partial DN map of \cite{KSU}.
Notice that the derivation of the stability estimate of \cite{CDR1} when the domain of observation is illuminated by a point at infinity, was revisited and simplified in \cite{CKS1}. For the stability issue in the two-dimensional case we refer to \cite{BIY,NSa,Sa}.

All the above mentioned results were obtained in a bounded domain. It turns out that there is only a small number of mathematical papers dealing with inverse boundary measurements problems in an unbounded domain. Several of them are concerned with the slab geometry. This is precisely the case of \cite{Ik,SW}, where embedded objects are identified in an infinite slab. In \cite{LU}, Li and Uhlmann proved that the compactly supported electric potential of the stationnary Schr\"odinger operator can
be determined uniquely, when the Dirichlet and Neumann data are given either on the different boundary hyperplanes of the slab or on the same hyperplane.
This result was generalized to the case of a magnetic Schr\"odinger operator in \cite{KLU}. Let us mention that inverse boundary value problems in an infinite slab were addressed by Yang in \cite{Y} for bi-harmonic operators. Recently, several stability results were derived in \cite{CS, Ki, KPS1, KPS2, BKS} for non compactly supported coefficients inverse problems in an infinite waveguide with a bounded cross section. More specifically, we refer to \cite{KKS, CKS} for the analysis of inverse problems in the framework of a periodic cylindrical domain examined in this paper.

\subsection{Notations and admissible potentials}
In this subsection we introduce some basic notations used throughout the section and define the set of admissible potentials under consideration in this paper.

Let $Y$ be either $\omega$, $\pd \omega$ or $G'$.
For $r$ and $s$ in $\R$, we denote by $\cH^{r,s}(\R \times Y)$ the set $H^r(\R;H^s(Y))$. Evidently we write $\cH^{r,s}(\Omega)$ (resp., $\cH^{r,s}(\Gamma)$, $\cH^{r,s}(G)$) instead of
$\cH^{r,s}(\R \times \omega)$ (resp., $\cH^{r,s}(\R \times \pd \omega)$, $\cH^{r,s}(\R \times G')$). Although this notation is reminiscent of the one used by Lions and Magenes in \cite{LM1} for anisotropic Sobolev spaces $H^r(\R;L^2(Y)) \cap L^2(\R;H^s(Y))$, it is worth noticing that they do not coincide with $\cH^{r,s}(\R \times Y)$, unless we have $r=s=0$. 
Next, it is easy to see for each $r>0$ and $s>0$ that $\cH^{-r,-s}(\R \times Y)$ is canonically identified with the space dual to $\cH_0^{r,s}(\R \times Y)$, with respect to the pivot space
$\cH^{0,0}(\R \times Y)=L^2(\R \times Y)$. Here we have set $\cH_0^{r,s}(\R \times Y):=H^r(\R;H_0^s(Y))$, where $H_0^s(Y)$ denotes the closure of $C_0^{\infty}(Y)$ in the topology of the Sobolev space $H^s(Y)$. 

Further, $X_1$ and $X_2$ being two Hilbert spaces, we denote by $\B(X_1,X_2)$ the class of bounded operators $T : X_1 \to X_2$. 

Let us now introduce the set of admissible unknown potentials. To this end we denote by $C_\omega$ the Poincar\'e constant associated with $\omega$, i.e. the largest of those constants $c>0$ such that the Poincar\'e inequality 
\bel{inegp1}
\| \nabla' u \|_{L^2(\omega)} \geq c \| u \|_{L^2(\omega)},\ u \in H_0^1(\omega),
\ee
holds. Here $\nabla'$ stands for the gradient with respect to $x'=(x_2,x_3)$. Otherwise stated, we have
\bel{p1const}
C_\omega := \sup \{ c >0\ \mbox{satisfying}\ \eqref{inegp1} \}.
\ee
For $M_- \in (0,C_\omega)$ and $M_+ \in [M_-,+\infty)$, we define the set of admissible unknown potentials as
\bel{admpot}
\mathscr{V}_\omega(M_\pm) := \{ V \in L^{\infty}(\Omega;\R)\ \mbox{satisfying}\ \eqref{eq-per},\ \| V \|_{L^{\infty}(\Omega)} \leq M_+\ \mbox{and}\ \| \max(0,-V) \|_{L^{\infty}(\Omega)} \leq M_- \}.
\ee
Notice that the constraint $\| \max(0,-V) \|_{L^{\infty}(\Omega)} \leq M_-$, imposed on admissible potentials $V$ in $\mathscr{V}_\omega(M_\pm)$, guarantees that the perturbation
by $V$ of the Dirichlet Laplacian in $\Omega$, is boundedly invertible in $L^2(\Omega)$, with norm not greater than $(C_\omega-M_-)^{-1}$. 
This condition could actually be weakened by only requiring that the distance of the spectrum of this operator to zero, be positive. Nevertheless, since the above mentioned condition on $V$ is more explicit than this latter, we stick with the definition \eqref{admpot} in the remaining part of this text.

\subsection{Statement of the main result}
Prior to stating the main result of this article we first examine in Proposition \ref{pr-a} below, the well-posedness of the BVP \eqref{eq1} in the space $H_\Delta(\Omega):=\{ u \in L^2(\Omega);\ \Delta u \in L^2(\Omega)\}$ endowed with the norm
$$ \norm{u}^2_{H_\Delta(\Omega)} :=\norm{u}_{L^2(\Omega)}^2+\norm{\Delta u}_{L^2(\Omega)}^2, $$
for suitable non-homogeneous Dirichlet boundary data $f$. Second, we rigorously define the DN map $\Lambda_V$ expressed in \eqref{a1} and describe its main properties.

As a preamble, we introduce the two following trace maps by adapting the derivation of \cite[Section 2, Theorem 6.5]{LM1}. Namely,
since $C_0^\infty(\overline{\Omega}):= \{ u_{| \overline{\Omega}},\ u \in C_0^{\infty}(\R^3) \}$ is dense in $H_\Delta(\Omega)$, by Lemma \ref{density} below, we extend the mapping
$$ \mathcal T_0 u :=u_{\vert\Gamma}\ (\mbox{resp.,}\ \mathcal T_1 u :={\pd_\nu u}_{\vert\Gamma}),\ u \in C_0^\infty(\overline{\Omega}), $$
into a continuous function $\mathcal T_0 : H_\Delta(\Omega) \to \cH^{-2,-\frac{1}{2}}(\Gamma)$ (resp., $\mathcal T_1 : H_\Delta(\Omega) \to \cH^{-2,-\frac{3}{2}}(\Gamma)$). We refer to Lemma \ref{trace} and its proof, for more details. 

Next we consider the space 
$$\sH(\Gamma):= \mathcal T_0 H_\Delta(\Omega) = \{ \mathcal T_0 u;\ u \in H_\Delta(\Omega) \},$$ 
and notice from Lemma \ref{p5} that $\mathcal T_0$ is bijective from $B:=\{ u \in L^2(\Omega);\ \Delta u = 0 \}$ onto $\sH(\Gamma)$. Therefore, with reference to \cite{BU,NS}, we put
\bel{nhg}
\norm{f}_{\sH(\Gamma)} : =\norm{\mathcal T_0^{-1} f}_{H_\Delta(\Omega)} = \norm{\mathcal T_0^{-1} f}_{L^2(\Omega)},
\ee
where $\mathcal T_0^{-1}$ denotes the operator inverse to $\mathcal{T}_0 : B \to \sH(\Gamma)$. 

We have the following existence and uniqueness result for the BVP \eqref{eq1}.

\begin{proposition}
\label{pr-a}
Pick $V \in  \mathscr{V}_\omega(M_\pm)$, where $M_- \in (0,C_\omega)$ and $M_+ \in [M_-,+\infty)$ are fixed.
\begin{enumerate}[(i)]
\item Then, for any $f \in \sH(\Gamma)$, there exists a unique solution $v \in L^2(\Omega)$
to \eqref{eq1}, such that the estimate
\bel{p6a}
\norm{v}_{L^2(\Omega)}\leq C \norm{f}_{\sH(\Gamma)},
\ee
holds for some constant $C>0$ depending only on $\omega$ and $M_\pm$. 
\item The DN map $\Lambda_V : f \mapsto \mathcal T_1 v_{| G}$ is a bounded operator from $\sH(\Gamma)$ into $\cH^{-2,-\frac{3}{2}}(G)$.
\item
Moreover, for each $W \in \mathscr{V}_\omega(M_\pm)$, the operator $\Lambda_{V}-\Lambda_{W}$ is bounded from $\sH(\Gamma)$ into $L^2(G)$.
\end{enumerate}
\end{proposition}

Put $\check{\Omega}:=(0,1) \times \omega$. In view of Proposition \ref{pr-a}, we now state the main result of this paper.  
\begin{theorem}
\label{thm1} 
Given $M_- \in (0,C_\omega)$ and $M_+ \in [M_-,+\infty)$, let $V_j \in \mathscr{V}_\omega(M_\pm)$ for $j=1,2$.  
Then, there exist two constants $C>0$ and $\gamma_* \in (0,1)$, both of them depending only on $\omega$, $M_\pm$ and $G'$, such that the estimate
\bel{thm1a} 
\norm{V_1-V_2}_{H^{-1}(\check{\Omega})}\leq C \Phi \left( \norm{\Lambda_{{V}_1}-\Lambda_{{V}_2}}\right),
\ee
holds for
\bel{def-Phi}
\Phi(\gamma) := \left\{ \begin{array}{cl} \gamma & \mbox{if}\ \gamma \geq \gamma^*,\\
(\ln \abs{\ln \gamma})^{-1} & \mbox{if}\ \gamma \in (0, \gamma^*),\\ 
0 & \mbox{if}\ \gamma=0.
\end{array} \right.
\ee
Here $\norm{\Lambda_{{V}_1}-\Lambda_{{V}_2}}$ denotes the norm of $\Lambda_{{V}_1}-\Lambda_{{V}_2}$ in $\B(\sH(\Gamma),L^2(G))$.
\end{theorem}

The statement of Theorem \ref{thm1} remains valid for any periodic potential $V \in L^{\infty}(\Omega)$, provided $0$ is in the resolvent set of $A_V$, the self-adjoint realization  in $L^2(\Omega)$ of the Dirichlet Laplacian $-\Delta + V$. In this case, the multiplicative constants $C$ and $\gamma_*$, appearing in \eqref{thm1a}-\eqref{def-Phi}, depend on (the inverse of) the distance $d>0$, between $0$ and the spectrum of $A_V$. In the particular case where $V \in \mathscr{V}_\omega(M_\pm)$, with $M_- \in (0,C_\omega)$, we have $d \geq C_\omega- M_-$, and the implicit condition $d>0$ imposed on $V$, can be replaced by the explicit one on the negative part of the potential, i.e. $\| \max(0,-V) \|_{L^\infty(\Omega)} \leq M_-$.

\subsection{Outline}
The remaining part of the paper is organized as follows. Section \ref{sec-pra} contains the proof of Proposition \ref{pr-a}. In Section \ref{sec-dec}, we decompose \eqref{eq1} with the aid of the Floquet-Bloch-Gel'fand (FBG) transform, into a family of BVP with quasi-periodic boundary conditions, 
\bel{eq2}
\left\{
\begin{array}{rcll}
(-\Delta  +V ) v & = & 0, & \mbox{in}\ \check{\Omega}, \\
v & = & g, & \mbox{on}\ \check{\Gamma}:=(0,1) \times \pd \omega, \\
v(1,\cdot ) - e^{i\theta} v(0,\cdot) & = & 0, & \mbox{in}\ \omega, \\
\pd_{x_1} v(1,\cdot) - e^{i\theta} \pd_{x_1} v(0,\cdot) & = & 0, & \mbox{in}\ \omega,
\end{array}
\right.
\ee
indexed by the real parameter $\theta \in [0,2 \pi)$. Here $g$ stands for the FBG transform of $f$, computed at $\theta$. We study the direct problem associated with \eqref{eq2} and reformulate the inverse problem under consideration as to whether the unknown function $V$ may be stably determined from the partial DN map associated with \eqref{eq2}, for any arbitrary $\theta \in [0, 2 \pi)$. We state in Theorem \ref{thm2} that the answer is positive and establish that this claim entails Theorem \ref{thm1}. There are two key ingredients in the proof of Theorem \ref{thm2}. The first one is a sufficiently rich set of suitable complex geometric optics (CGO) solutions to \eqref{eq2}, built in Section \ref{sec-cgo}. The second one is a specifically designed Carleman estimate for quasi-periodic Laplace operators, derived in Section \ref{sec-ce}. Finally, the proof of Theorem \ref{thm2} is displayed in Section \ref{sec-thm2}.

Let us now briefly comment on the strategy of the proof of Theorems \ref{thm1} and \ref{thm2}. Our approach is similar to the one of \cite{HW} as it combines CGO solutions to the quasi-periodic Laplace equation in $\check{\Omega}$ with a suitable Carleman estimate. Nevertheless, in contrast to \cite{CDR1, CDR2}, the Carleman estimate of \cite[Proposition 3.2]{KSU} is not adapted to the framework of this paper. This is due to the quasi-periodic boundary conditions imposed on the CGO solutions employed in the context of inverse periodic coefficients problems. Therefore, in view of taking the Neumann measurements on $G$ only, we shall rather use the Carleman estimate with linear weights introduced in \cite{BU}. 

\section{Proof of Proposition \ref{pr-a}}
\label{sec-pra}

In this section we prove the claim of Proposition \ref{pr-a}. As a preliminary we introduce in Subsection \ref{sec-to} the trace operators $\mathcal{T}_j$, $j=0,1$, and establish some useful properties that are needed for the proof of Proposition \ref{pr-a}, which can be found in Subsection \ref{sec-compproofpr-a}.

\subsection{The trace operators}
\label{sec-to}

The rigorous definition of the trace operators $\mathcal{T}_j$, $j=0,1$, boils down to the coming lemma. Such a density result is rather classical for bounded domains (see e.g. \cite[Section 2, Theorem 6.4]{LM1}), but it has to be justified here since $\Omega$ is infinitely extended in the $x_1$ direction.

\begin{lemma}
\label{density} 
The space $C_0^\infty(\overline{\Omega})$ is dense in $H_\Delta(\Omega)$.
\end{lemma}
\begin{proof} 
Let $f \in H_\Delta(\Omega)'$, the space of linear continuous forms on $H_\Delta(\Omega)$, satisfy
\bel{d1}
\langle f , w \rangle_{H_\Delta(\Omega)',H_\Delta(\Omega)}=0,\ w \in C_0^\infty(\overline{\Omega}).
\ee
In order to establish the claim of Lemma \ref{density}, it is enough to show that $f$ is identically zero.

To do that we put 
$\jmath : u \in H_\Delta(\Omega) \mapsto (u , \Delta u) \in L^2(\Omega)^2$, notice that $H_\Delta(\Omega)$ is isometrically isomorphic to the closed subspace $Y:=\jmath ( H_\Delta(\Omega) )$ of $L^2(\Omega)^2$, and introduce the following linear continuous form $g$ on $Y$:
$$ \langle g , v \rangle_{Y',Y} := \langle f , \jmath^{-1} v \rangle_{H_\Delta(\Omega)', H_\Delta(\Omega)},\ v \in Y. $$ 
Here and henceforth, $Y'$ denotes the space dual to $Y$ and $\langle \cdot , \cdot \rangle_{Y',Y}$ stands for the duality pairing between $Y'$ and $Y$.
Since $g$ can be extended by Hahn Banach theorem to a linear continuous form $\tilde{g}$ on $L^2(\Omega)^2$, we may find $(g_1,g_2) \in L^2(\Omega)^2$ such that
\bel{d2}
\langle f , u \rangle_{H_\Delta(\Omega)',H_\Delta(\Omega)} = \langle \tilde{g} , \jmath u \rangle_{L^2(\Omega)^2} = \langle g_1 , u \rangle_{L^2(\Omega)} + \langle g_2 , \Delta u \rangle_{L^2(\Omega)},\ u \in H_\Delta(\Omega),
\ee
according to Riesz representation theorem. Upon extending $g_1$ and $g_2$ by zero outside $\Omega$, we deduce from \eqref{d2} that
$$
\int_{\R^3} (g_1 \overline{w} + g_2 \overline{\Delta w} ) dx = 0,\ w \in C_0^\infty(\R^3),
$$
whence 
\bel{d3}
-\Delta g_2 = g_1\ \mbox{in}\ \R^3.
\ee 
Since $g_1 \in L^2(\Omega)$, then \eqref{d3} yields that $g_2 \in H^2(\Omega)$ by the classical elliptic regularity property. Further, as $g_2$ vanishes in $\R^2 \setminus \overline{\Omega}$, we get that $g_2 \in H_0^2(\Omega)$, the closure of $C_0^{\infty}(\Omega)$ in the topology of the second-order Sobolev space $H^2(\Omega)$. As a consequence we have
$$ \langle g_2 , \Delta u \rangle_{L^2(\Omega)} = \langle \Delta g_2 , u \rangle_{L^2(\Omega)} = - \langle g_1 , u \rangle_{L^2(\Omega)}, \ u \in H_\Delta(\Omega), $$
by \eqref{d3}. In view of \eqref{d2} this entails that $f=0$, hence the result.
\end{proof}

Armed with Lemma \ref{density}, we now define the trace maps $\mathcal T_j$, $j=0,1$, on the space $H_\Delta(\Omega)$, in the following manner.

\begin{lemma}
\label{trace}
The mapping $w \in C_0^\infty(\overline{\Omega}) \mapsto w_{\vert \pd \Omega}$ (resp., $w \in C_0^\infty(\overline{\Omega}) \mapsto \pd_\nu w_{\vert \pd \Omega}$) can be extended over $H_\Delta(\Omega)$ to a bounded operator $\mathcal T_0 : H_\Delta(\Omega) \to \cH^{-2,-\frac{1}{2}}(\Gamma)$ (resp., $\mathcal T_1 : H_\Delta(\Omega) \to \cH^{-2,-\frac{3}{2}}(\Gamma)$).
\end{lemma}
\begin{proof} It is well known that $u \in C^{\infty}(\overline{\omega}) \mapsto (u_{\vert \pd \omega}, {\pd_\nu u}_{\vert \pd \omega})$ extends continuously to a bounded operator from $H^2(\omega)$ onto $H^{\frac{3}{2}}(\pd \omega) \times H^{\frac{1}{2}}(\pd \omega)$, so there exists $L : H^{\frac{3}{2}}(\pd \omega) \times H^{\frac{1}{2}}(\pd \omega) \to H^2(\omega)$, linear and bounded, such that 
$$
L(h_1,h_2)_{\vert \pd \omega}=h_1,\ {\pd_\nu L(h_1,h_2)}_{\vert \pd \omega}=h_2,\ (h_1,h_2) \in H^{\frac{3}{2}}(\pd \omega) \times H^{\frac{1}{2}}(\pd \omega).
$$
Let us define the operator $\mathcal L : \cH^{2,\frac{3}{2}}(\Gamma) \times \cH^{2,\frac{1}{2}}(\Gamma) \to H^2(\Omega)$, by setting for a.e. $(x_1,x') \in \Omega$,
$$  \mathcal L(h_1,h_2) (x_1,x') := L \left( h_1(x_1,\cdot),h_2(x_1,\cdot) \right) (x'),\ (h_1,h_2) \in \cH^{2,\frac{3}{2}}(\Gamma) \times \cH^{2,\frac{1}{2}}(\Gamma). $$
Using that $\| f \|_{H^2(\Omega)}^2=\sum_{k=0}^2 \| \pd_{x_1}^k f \|_{L^2(\R,H^{2-k}(\omega))}^2$, it is easy to see that $\mathcal L$ is bounded, and we check for every $(h_1,h_2) \in \cH^{2,\frac{3}{2}}(\Gamma) \times H^{2,\frac{1}{2}}(\Gamma)$ that
$$ \mathcal L(h_1,h_2)_{\vert \pd \Omega}=h_1,\  {\pd_\nu L(h_1,h_2)}_{\vert\pd \Omega}=h_2. $$

For $h \in \cH^{2,\frac{1}{2}}(\Gamma)$, we put $w:=\mathcal L(0,h)$ in such a way that $w \in H^2(\Omega)$ satisfies
\bel{trace1}
w_{\vert \pd \Omega}=0,\ \pd_\nu w_{\vert \pd \Omega}=h\ \mbox{and}\ \norm{w}_{H^2(\Omega)} \leq C \norm{h}_{\cH^{2,\frac{1}{2}}(\Gamma)}, 
\ee
the constant $C>0$ being independent of $h$. Next, for notational simplicity, we denote by $\mathcal T_0 v$ the trace $v_{| \pd \Omega}$ of any function $v \in C_0^\infty(\overline{\Omega})$. Then, applying Green formula twice with respect to $x' \in \omega$, and integrating by parts with respect to $x_1$ over $\R$, we find that
$$ \int_{\pd \Omega} (\mathcal T_0 v) \overline{h} d\sigma(x) = \langle v , \Delta w \rangle_{L^2(\Omega)} - \langle \Delta v , w \rangle_{L^2(\Omega)},\ v \in C_0^\infty(\overline{\Omega}). $$
Therefore, we get $\abs{\int_{\pd \Omega} (\mathcal T_0 v) \overline{h} d\sigma(x)} \leq 2 \norm{v}_{H_\Delta(\Omega)} \norm{w}_{H^2(\Omega)}$ by Cauchy-Schwarz inequality, and hence
$$ \abs{\langle \mathcal T_0 v , h \rangle_{\cH^{-2,-\frac{1}{2}}(\Gamma) , \cH^{2,\frac{1}{2}}(\Gamma)}} \leq 2C \norm{v}_{H_\Delta(\Omega)} \norm{h}_{\cH^{2,\frac{1}{2}}(\Gamma)}, $$
with the help of \eqref{trace1}. Since $h$ is arbitrary in $\cH^{2,\frac{1}{2}}(\Gamma)$, this entails that
$$ \norm{\mathcal T_0 v}_{\cH^{-2,-\frac{1}{2}}(\Gamma)}\leq 2 C\norm{v}_{H_\Delta(\Omega)},\ v \in C_0^\infty(\overline{\Omega}), $$
which together with Lemma \ref{density}, proves that $\mathcal T_0$ can be extended over $H_\Delta(\Omega)$ to a bounded operator into $\cH^{-2,-\frac{1}{2}}(\Gamma)$. 

To prove the second part of the claim, we pick $g\in \cH^{2,\frac{3}{2}}(\Gamma)$ and set $w:=\mathcal L(g,0)$, so we have $w \in H^2(\Omega)$ and
\bel{trace2}
w_{\vert \pd \Omega}=g,\ {\pd_\nu w}_{\vert \pd \Omega}=0\ \mbox{and} \norm{w}_{H^2(\Omega)} \leq C \norm{g}_{\cH^{2,\frac{3}{2}}(\Gamma)},
\ee
for some constant $C>0$ that is independent of $w$. Next we denote by $\mathcal T_1 v$ the trace $\pd_\nu v_{| \pd \Omega}$ of any $v \in C_0^\infty(\overline{\Omega})$ on $\pd \Omega$. Then, arguing as before, we find that the following estimate
$$
\abs{\langle \mathcal T_1 v , g \rangle_{\cH^{-2,-\frac{3}{2}}(\Gamma),\cH^{2,\frac{3}{2}}(\Gamma)}} \leq 2C \norm{v}_{H_\Delta(\Omega)}\norm{g}_{\cH^{2,\frac{3}{2}}(\Gamma)}, 
$$
holds uniformly in $g \in \cH^{2,\frac{3}{2}}(\Gamma)$ and $v \in C_0^\infty(\overline{\Omega})$. This and Lemma \ref{density} yield that
$\mathcal T_1$ is extendable to a linear bounded operator from $H_\Delta(\Omega)$ into $\cH^{-2,-\frac{3}{2}}(\Gamma)$.
\end{proof}

We end this subsection by establishing the following result which was required by \eqref{nhg} to define the topology of the space $\sH (\Gamma)$.

\begin{lemma}
\label{p5} 
$\mathcal T_0$ is bijective from $B=\{ u \in L^2(\Omega);\ \Delta u = 0 \}$ onto $\cH^{-2,-\frac{1}{2}}(\Gamma)$.
\end{lemma}
\begin{proof} 
Given $u$ and $v$ in $B$ obeying $\mathcal T_0 u = \mathcal T_0 v$, we see that $w:=u-v$ satisfies the system
\bel{ea5}
\left\{ \begin{array}{rcll} 
-\Delta w & = & 0, & \mbox{in}\ \Omega,\\ 
w & = & 0,& \mbox{on}\ \Gamma. \end{array} 
\right.
\ee
Therefore we have $w=0$, by uniqueness (see e.g. \cite[Lemma 2.4]{KPS2}) of the solution to \eqref{ea5}, proving that $\mathcal T_0$ is injective on $B$.

Further, we know from the definition of $\sH(\Gamma)$ that for any $f \in \sH(\Gamma)$, there exists $u \in H_\Delta(\Omega)$ such that $\mathcal T_0 u=f$. Next, since $\Delta u \in L^2(\Omega)$, then the BVP
$$
\left\{ \begin{array}{rcll} -\Delta v & = & \Delta u, & \mbox{in}\ \Omega,\\ 
v & = & 0,& \mbox{on}\ \Gamma, \end{array}
\right.
$$
admits a unique solution $v \in H^2(\Omega) \cap H^1_0(\Omega)$. Now, it is apparent that $w= u + v\in L^2(\Omega)$ satisfies $\Delta w=0$ and 
$\mathcal T_0 w =\mathcal T_0 u +\mathcal T_0 v=f$.
\end{proof}

\subsection{Completion of the proof}
\label{sec-compproofpr-a}
Firstly, we introduce $A_V$, the self-adjoint operator in $L^2(\Omega)$, generated by the closed quadratic form 
$$ u \mapsto a_V[u]:= \int_{\Omega} ( | \nabla u(x) |^2 + V(x) | u(x) |^2 ) dx,\ u \in D(a_V):=H_0^1(\Omega). $$ 
We recall from \cite[Lemma 2.2]{CKS} that $A_V$ acts as $-\Delta+V$ on its domain $D(A_V)= H^1_0(\Omega) \cap H^2(\Omega)$.
Moreover, since $V \in \mathscr{V}_\omega(M_\pm)$ we have
$$ a_V[u] \geq (C_\omega-M_-) \| u \|_{L^2(\Omega)}^2,\ u \in H_0^1(\Omega), $$
by \eqref{p1const}-\eqref{inegp1}, hence $A_V$ is boundedly invertible in $L^2(\Omega)$ and
\bel{binv}
\| A_V^{-1} \|_{\B (L^2(\Omega))} \leq \frac{1}{C_\omega-M_-}.
\ee

In order to establish the first statement of Proposition \ref{pr-a}, i.e. (i), we notice that $v$ is solution to \eqref{eq1} if and only if
$u:=v-\mathcal{T}_0^{-1} f$ solves the system
\bel{eq1b}
\left\{ 
\begin{array}{rcll} 
(-\Delta + V) u & = & -V \mathcal T_0^{-1} f & \mbox{in}\ \Omega,\\ 
u & = & 0 & \mbox{on}\ \Gamma.
\end{array}
\right.
\ee
Since $V \mathcal T_0^{-1}f \in L^2(\Omega)$ and $A_V$ is boundedly invertible in $L^2(\Omega)$, then 
$u:=-A_V^{-1} V \mathcal T_0^{-1} f$ is the unique solution to \eqref{eq1b}. As a consequence we have
\bel{p6b}
\norm{u}_{L^2(\Omega)} \leq M_+ \norm{A_V^{-1}}_{\B (L^2(\Omega))} \norm{\mathcal T_0^{-1} f}_{L^2(\Omega)}.
\ee
Evidently $v:=u+\mathcal T_0^{-1} f$ is the unique solution to \eqref{eq1}, and \eqref{p6a} follows readily from this, \eqref{nhg}, \eqref{binv} and \eqref{p6b}. 

We turn now to proving (ii)-(iii). For $f \in \sH(\Gamma)$ fixed, we still denote by $v$ the solution to \eqref{eq1} associated with $f$. Since $v \in L^2(\Omega)$ and $\Delta v= V v \in L^2(\Omega)$ it holds true that $v \in H_\Delta(\Omega)$ and that
$$ \norm{v}_{H_\Delta(\Omega)}^2=\norm{v}_{L^2(\Omega)}^2+\norm{V v}_{L^2(\Omega)}^2 \leq (1+M_+^2)\norm{v}_{L^2(\Omega)}^2. $$
From the continuity of $\mathcal T_1 : H_\Delta(\Omega) \to \cH^{-2,-\frac{3}{2}}(\Gamma)$, it then follows that
$$\norm{\mathcal T_1 v}^2_{\cH^{-2,-\frac{3}{2}}(\Gamma)} \leq (1+M_+^2) \| \mathcal T_1 \|_{\mathcal{B}(H_\Delta(\Omega), \cH^{-2,-\frac{3}{2}}(\Gamma))}^2 \norm{v}_{L^2(\Omega)}^2. $$ 
This and \eqref{p6a} yield that the DN map $\Lambda_V$ is bounded from $\sH(\Gamma)$ into $\cH^{-2,-\frac{3}{2}}(G)$.

Let us now establish (iii). We denote by $w$ the solution to \eqref{eq1} where $W$ is substituted for $V$, and notice that
$u:=v-w$ is solution to the BVP
$$
\left\{ 
\begin{array}{rcll} 
(-\Delta + V) u & = & (W-V) w & \mbox{in}\ \Omega\\ 
u & = & 0 & \mbox{on}\ \Gamma.
\end{array}
\right.
$$
Since $(W-V) w \in L^2(\Omega)$ and $0$ is in the resolvent set of $A_V$, we have 
\bel{ea0}
u=A_V^{-1} (W-V) w,
\ee
whence $u \in D(A_V) = H^2(\Omega) \cap H^1_0(\Omega)$. As the usual norm in $H^2(\Omega)$ is equivalent to the one associated with the domain of $A_V$, by \cite[Lemma 2.2]{CKS}, we have
$$
\norm{u}_{H^2(\Omega)} 
\leq C \left( \norm{A_{V} u}_{L^2(\Omega)} + \| u \| _{L^2(\Omega)} \right), 
$$
for some constant $C>0$, depending only on $\omega$ and $M_+$. This, \eqref{binv} and \eqref{ea0} yield that
\bel{ea1}
\norm{u}_{H^2(\Omega)} \leq C \left( 1 + \| A_V^{-1} \|_{\B(L^2(\Omega))} \right) \norm{(W-V) w}_{L^2(\Omega)} \leq 2 M_+ C \left( 1 + \frac{1}{C_\omega-M_-} \right) \norm{w}_{L^2(\Omega)}.
\ee
Bearing in mind that $\| w \|_{L^2(\Omega)} \leq c \| f \|_{\sH (\Gamma)}$, by \eqref{p6a}, and using the continuity of the trace operator $u \mapsto {\pd_\nu u}_{| \Gamma}$ from $H^2(\Omega)$ into $L^2(\Gamma)$, we deduce from \eqref{ea1} that
$$ \norm{ \pd_\nu u}_{ L^2(\Gamma)} \leq C \norm{u}_{H^2(\Omega)} \leq C \norm{f}_{\sH(\Gamma)}, $$
where $C$ is another positive constant depending only on $\omega$ and $M_\pm$. Now, the desired result follows from this and the identity
$\pd_\nu u_{|G}=\mathcal T_1 u_{|G}={\mathcal T_1 v}_{|G}-{\mathcal T_1 w}_{|G}=(\Lambda_{V}-\Lambda_{W})f$.

\section{Fiber decomposition}
\label{sec-dec}

In this section we decompose \eqref{eq1} into the family of BVP \eqref{eq2} indexed by $\theta \in [0, 2 \pi)$. This is by means of the FBG transform, introduced in Subsection \ref{sec-fbg}, which decomposes the operator $A_V$. In Subsection \ref{sec-fib} we examine the direct problem associated with \eqref{eq2} for each $\theta \in [0, 2 \pi)$ and study the corresponding DN map $\Lambda_{V,\theta}$. Finally, in Subsection \ref{sec-lin} we reformulate the inverse problem under consideration as to whether $V$ can be stably retrieved from partial knowledge of $\Lambda_{V,\theta}$. We state in Theorem \ref{thm2} that this is actually the case provided the Neumann data are measured on $\check{G}:=(0,1) \times G'$.

\subsection{FBG transform and fiber decomposition of $A_V$}   
\label{sec-fbg}
Let $Y$ be either $\omega$ or $\pd \omega$.
The main tool for the analysis of 1-periodic waveguides $\R \times Y$ is the partial FBG transform defined for every 
$f \in C_0^\infty (\R \times Y)$ as
\bel{c1}
(\mathcal{U}_Y f)_{\theta} (x_1,y):=\sum_{k \in \Z} e^{-i k \theta} f(x_1+k,y),\ (x_1,y) \in \R \times Y,\ \theta \in [0,2 \pi).
\ee
For notational simplicity, we systematically drop the $Y$ in $\mathcal{U}_Y$ and write $\mathcal{U}$ instead of $\mathcal{U}_Y$.
In view of \cite[Section XIII.16]{RS2}, the above operator can be extended to a unitary operator, still denoted by $\mathcal{U}$, from $L^2(\R \times Y)$ onto the Hilbert space 
defined as the direct integral sum $\int_{(0,2 \pi)}^\oplus L^2((0,1) \times Y) \frac{d \theta}{2 \pi}$. Since $\check{\Omega}=(0,1) \times \omega$ (resp., $\check{\Gamma}=(0,1) \times \pd \omega$), then $\mathcal{U}$ maps $L^2(\Omega)$ onto $\int_{(0,2 \pi)}^\oplus L^2(\check{\Omega}) \frac{d \theta}{2 \pi}$ if $Y=\omega$, (resp., $L^2(\pd \Omega)$ onto $\int_{(0,2 \pi)}^\oplus L^2(\check{\Gamma}) \frac{d \theta}{2 \pi}$ if $Y=\pd \omega$). 

Let $V \in L^\infty(\Omega;\R)$ satisfy \eqref{eq-per}. Since $V$ is $1$-periodic with respect to $x_1$, the operator $A_V$ defined in Section \ref{sec-intro}, is decomposed by $\mathcal{U}$. To make this claim more precise, we first introduce the following functional spaces.
For $\theta \in [0,2\pi)$ fixed, we put with reference to \cite[Section 6.1]{CKS}, 
$$
\cH_{\theta}^s((0,1)\times Y) := \left\{ u \in H^s((0,1)\times Y);\ \pd_{x_1}^j u(1,\cdot) - e^{i\theta} \pd_{x_1}^j u(0,\cdot) =0,\ j < s-\frac{1}{2} \right\}\ \mbox{if}\ s>\frac{1}{2},$$
and 
$$ \cH_{\theta}^s((0,1)\times Y) := H^s((0,1)\times Y)\ \mbox{if}\ s \in \left[0,\frac{1}{2} \right]. $$ 
Evidently, we shall write $\cH_{\theta}^s(\check{\Omega})$ (resp., $\cH_{\theta}^s(\check{\Gamma})$) instead of $\cH_{\theta}^s((0,1)\times Y)$ for $Y=\omega$ (resp., $Y=\pd \omega$).

Next we introduce the operator
$$ \mathcal A_{V,\theta}:=-\Delta + V,\ D(\mathcal A_{V,\theta}):=\mathcal U D(A_V) = \cH_{\theta}^2(\check{\Omega}) \cap L^2(0,1;H_0^1(\omega)), $$
self-adjoint in $L^2(\check{\Omega})$.
Then, taking into account that
$$
(\mathcal U \psi)_\theta(x_1+1,x')=e^{i \theta} ( \mathcal U \psi)_\theta(x_1,x'),\ x_1 \in \R,\ x' \in \omega,\ \theta \in [0,2 \pi),
$$
and that
\bel{c3}
\left( \mathcal U \frac{\pd^m \psi}{\pd x_j^m} \right)_\theta=\frac{\pd^m (\mathcal U \psi)_\theta}{\pd x_j^m},\ j=1,2,3,\ m \in \N^*,\ \theta \in [0,2 \pi),
\ee
we find that $(\mathcal U A_V \psi)_\theta = \mathcal A_{V,\theta} ( \mathcal U \psi )_\theta$ for any $\psi \in D(A_V)$, i.e.
\bel{ec1}
\mathcal U A_V \mathcal U^{-1}= \int_{(0,2 \pi)}^\oplus \mathcal A_{V,\theta} \frac{d \theta}{2 \pi}.
\ee

Having seen this we turn now to analysing the direct problem associated with \eqref{eq2}.

\subsection{Analysis of the fibered problems}
\label{sec-fib}
For $\theta \in [0, \pi)$ and $n \in \N \cup \{ \infty \}$, we set
$$ \cC_{\theta}^n \left( [0,1]\times\overline{\omega}\right):=\{u \in \cC^n\left([0,1] \times \overline{\omega}\right);\ \pd_{x_1}^j u(1,\cdot) - e^{i\theta} \pd_{x_1}^j u(0,\cdot) =0\ \mbox{for all}\ j \in \N\ \mbox{such that}\ j \leq n\}.$$
A slight modification of the proof of \cite[Section 2, Theorem 6.4]{LM1} shows that $C^\infty\left([0,1]\times\overline{\omega}\right)$ is dense in $H_\Delta(\check{\Omega}):=\{u\in L^2(\check{\Omega});\ \Delta u\in L^2(\check{\Omega})\}$ endowed with the norm
$\norm{u}_{H_\Delta(\check{\Omega})}:=\left( \norm{u}_{L^2(\check{\Omega})}^2+\norm{\Delta u}_{L^2(\check{\Omega})}^2 \right)^\frac{1}{2}$.
\begin{lemma}
\label{np} 
For $j=0$ and $j=1$, each of the two following mappings 
$$ u \mapsto \pd_{x_1}^j u(0,\cdot)\ \mbox{and}\ u \mapsto \pd_{x_1}^j u(1,\cdot),\ u \in C^\infty \left( [0,1] \times \overline{\omega} \right), 
$$
can be extended to a bounded operator from $H_\Delta(\check{\Omega})$ into $H^{-2}(\omega)$.
\end{lemma}
\begin{proof} 
Let us prove the claim for $u \mapsto \pd_{x_1} u(0,\cdot)$, the three other cases being treated in the same way.
To this purpose we consider a function $\chi \in C_0^{\infty}(\R)$ satisfying $\chi(x_1)=1$ for $x_1 \in [-1 \slash 3, 1 \slash 3]$, $\chi(x_1) \in [0,1]$ for $x_1 \in (1 \slash 3,2 \slash 3)$ and $\chi(x_1)=0$ for $x_1 \in [2 \slash 3,4 \slash 3]$. We pick $g \in H_0^2(\omega)$, the closure of $C_0^\infty(\omega)$ in $H^2(\omega)$, put $v(x_1,x'):=\chi(x_1) g(x')$, and then notice that
\bel{np1} 
v=\pd_\nu v = 0\ \mbox{on}\ \check{\Gamma},\ \pd_{x_1} v(0,\cdot)=v(1,\cdot)=\pd_{x_1} v(1,\cdot) =0\ \mbox{and}\ v(0,\cdot)=g\ \mbox{in}\ \omega.
\ee
Moreover, there exists a positive constant $C=C(\omega,\chi)$ such that we have
\bel{np2}
\norm{v}_{H^2(\check{\Omega})} \leq C \norm{g}_{H^2(\omega)}.
\ee
For all $u \in C^{\infty}([0,1] \times \overline{\omega})$, we deduce from \eqref{np1} and the Green formula that
$$ \langle \pd_{x_1} u(0,\cdot) , g \rangle_{L^2(\omega)} = \langle u ,  \Delta v \rangle_{\check{\Omega}} - \langle \Delta u ,  v \rangle_{\check{\Omega}}. $$
This and \eqref{np2} yield
$$
\abs{\langle \pd_{x_1} u(0,\cdot) , g \rangle_{H^{-2}(\omega),H_0^2(\omega)}} \leq 2\norm{u}_{H_\Delta(\check{\Omega})} \norm{v}_{H^2(\check{\Omega})} \leq C \norm{u}_{H_\Delta(\check{\Omega})}\norm{g}_{H^2(\omega)}.
$$
Therefore, since $C^\infty\left([0,1] \times \overline{\omega}\right)$ is dense in $H_\Delta(\check{\Omega})$, we may extend the mapping $u \mapsto \pd_{x_1} u(0,\cdot)$ to a continuous map from $H_\Delta(\check{\Omega})$ into $H^{-2}(\omega)$. 
\end{proof}

With reference to Lemma \ref{np}, we introduce the following closed subset of $H_\Delta(\check{\Omega})$:
$$ 
H_{\Delta,\theta}(\check{\Omega}):=\{ u \in H_\Delta(\check{\Omega});\ \pd_{x_1}^j u(1,\cdot) - e^{i \theta} \pd_{x_1}^j u(0,\cdot) = 0\ \mbox{on}\ \omega\ \mbox{for}\  j=0,1 \}.
$$
Notice for further use that
\bel{ec1b}
\mathcal U H_\Delta(\Omega) = \int_{(0, 2 \pi)}^\oplus H_{\Delta,\theta}(\check{\Omega}) \frac{d \theta}{2 \pi},
\ee
and that $H_{\Delta,\theta}(\check{\Omega}) \cap C^\infty \left([0,1]\times\overline{\omega}\right) \subset \cH_{\theta}^2(\check{\Omega})$. Moreover, since
$\cH_{\theta}^2(\check{\Omega})$ is continuously embedded in $H_{\Delta,\theta}(\check{\Omega})$ and that $\cC_{\theta}^\infty \left([0,1]\times\overline{\omega}\right)$ is dense in $\cH_{\theta}^2(\check{\Omega})$, then the space $\cC_{\theta}^\infty\left([0,1] \times \overline{\omega}\right)$ is dense in $H_{\Delta,\theta}(\check{\Omega})$ as well. 
Therefore, by reasoning in the same way as in the proof of Lemma \ref{trace}, we extend the mapping $w \in \cC_{\theta}^\infty \left([0,1]\times \overline{\omega}\right) \mapsto w_{\vert\check{\Gamma}}$ (resp., $w \in \cC_{\theta}^\infty \left([0,1]\times \overline{\omega}\right) \mapsto \pd_\nu w_{\vert\check{\Gamma}}$) to a bounded operator 
$$ \mathcal T_{0,\theta} : H_{\Delta,\theta}(\check{\Omega}) \to \cH_{\theta}^{-2}(0,1; H^{-\frac{1}{2}}(\pd \omega))\ (\mbox{resp.},\ \mathcal T_{1,\theta} : H_{\Delta,\theta}(\check{\Omega}) \to \cH_{\theta}^{-2}(0,1; H^{-\frac{3}{2}}(\pd \omega))).
$$ 
Here we denote by $\cH_{\theta}^s(0,1;X)$, where $s \in \R$ and $X$ is a Banach space for the norm $\norm{\cdot}_X$, the set of functions
$$ 
t \in (0,1) \mapsto \varphi(t) := \sum_{k\in\Z} \varphi_{k} e^{i(\theta + 2 \pi k)t}\ \mbox{with}\
(\varphi_{k})_{k} \in X^{\Z}\ \mbox{satisfying}\ \sum_{k \in \Z }(1+k^2)^{s} \norm{\varphi_{k}}_{X}^2<\infty.
$$
Endowed with the norm $\norm{\varphi}_{\cH_{\theta}^s(0,1;X)}:=\left(\sum_{k \in \Z} (1+k^2)^{s} \norm{\varphi_{k}}_{X}^2 \right)^{\frac{1}{2}}$, $\cH_{\theta}^s(0,1;X)$ is a  
Banach space. Notice that if $X'$ is the dual space of $X$, then $\cH_{\theta}^{-s}(0,1;X')$ is the space dual to $\cH_{\theta}^s(0,1;X)$.

Let us next introduce the set
$$ \sH_\theta(\check{\Gamma}) :=\{ \mathcal T_{0,\theta} u;\ u \in H_{\Delta,\theta}(\check{\Omega}) \}. $$
Arguing as in the derivation of Lemma \ref{p5}, we check that $\mathcal T_{0,\theta}$ is a bijection from $B_\theta :=\{ u\in H_{\Delta,\theta}(\check{\Omega});\ \Delta u=0\}$ onto $\sH_\theta(\check{\Gamma})$. Put
\bel{ec2}
\norm{f}_{\sH_\theta(\check{\Gamma})} :=\norm{\mathcal T_{0,\theta}^{-1} f}_{H_{\Delta}(\check{\Omega})} = \norm{\mathcal T_{0,\theta}^{-1} f}_{L^2(\check{\Omega})},
\ee
where $\mathcal T_{0,\theta}^{-1}$ denotes the operator inverse to $\mathcal T_{0,\theta} : B_\theta \to \sH_\theta(\check{\Gamma})$.

We may now establish the following technical result, which is inspired by Proposition \ref{pr-a}.

\begin{proposition}
\label{p9} 
Pick $\theta \in [0,2\pi)$ and assume that the conditions of Proposition \ref{pr-a} are satisfied. Then the three following statements are true.
\begin{enumerate}[(i)]
\item For any $f \in \sH_\theta(\check{\Gamma})$, there exists a unique solution $v \in H_{\Delta,\theta}(\check{\Omega})$
to \eqref{eq2}, satisfying
\bel{p6z}
\norm{v}_{L^2(\check{\Omega})}\leq C \norm{f}_{\sH_\theta(\check{\Gamma})},
\ee
for some positive constant $C=C(\omega,M_\pm)$, independent of $\theta$. 
\item The DN map $\Lambda_{V,\theta} : f \mapsto \mathcal T_{1,\theta} u_{|\check{G}}$, where we recall that $\check{G}=(0,1) \times G'$, is a bounded operator from $\sH_\theta(\check{\Gamma})$ into $\cH_{\theta}^{-2}(0,1;H^{-\frac{3}{2}}(G'))$.
\item For every $W \in \mathscr{V}_\omega(M_\pm)$, the operator $\Lambda_{V,\theta}-\Lambda_{W,\theta}$ is bounded from $\sH_\theta(\check{\Gamma})$ into $L^2(\check{G})$.
\end{enumerate}
\end{proposition}

\begin{proof} 
To prove (i) we use the fact that $v$ is solution to \eqref{eq2} if and only if the function $u:=v-\mathcal T_{0,\theta}^{-1} f$  satisfies the BVP
\bel{Eq2}
\left\{ 
\begin{array}{rcll} 
(-\Delta + V) u & = & -V \mathcal T_{0,\theta}^{-1} f, & \mbox{in}\ \check{\Omega},\\ 
u & = & 0,& \mbox{on}\ \check{\Gamma}, \\
u(1,\cdot) - e^{i\theta}  u(0,\cdot) & = & 0, & \mbox{in}\ \omega. \\
\pd_{x_1} u(1,\cdot) - e^{i\theta} \pd_{x_1} u(0,\cdot) & = & 0, & \mbox{in}\ \omega.
\end{array} 
\right.
\ee
Due to \eqref{ec1} we know from \cite[Theorem XIII.98]{RS2} that $\sigma(\mathcal A_{V,\theta}) \subset \sigma(A_V)$. Hence $0 \notin \sigma(\mathcal A_{V,\theta})$ and \eqref{Eq2} admits a unique solution $u=- \mathcal A_{V,\theta}^{-1} V \mathcal T_{0,\theta}^{-1} f$ as $V \mathcal T_{0,\theta}^{-1} f \in L^2(\check{\Omega})$. As a consequence $v=(I-\mathcal A_{V,\theta}^{-1} V) \mathcal T_{0,\theta}^{-1} f$ is the unique solution to the BVP \eqref{eq2} and \eqref{p6z} follows readily from this, \eqref{ec2} and the estimate
$$ \| \mathcal A_{V,\theta}^{-1} \|_{\B(L^2(\check{\Omega}))} \leq \frac{1}{C_\omega-M_-}, $$
arising from \eqref{inegp1}-\eqref{p1const} and the fact that the operator $\mathcal A_{V,\theta}$ is generated by the following quadratic form:
\bel{qf}
u \mapsto a_{V,\theta}[u] := \int_{\check{\Omega}} \left( |\nabla u(x) |^2 + V(x) |u(x)|^2 \right) dx,\ u \in D(a_{V,\theta}):= \cH_{\theta}^1(\check{\Omega}) \cap L^2(0,1;H_0^1(\omega)).
\ee

The rest of the proof follows the same lines as the derivation of (ii)-(iii) in Proposition \ref{pr-a}.
\end{proof}

For further reference, we now establish for every $\theta \in [0,2 \pi)$ that the estimate
\bel{esti} 
\norm{\mathcal T_{0,\theta}u}_{\sH_{\theta}(\check{\Gamma})} \leq C\norm{u}_{H_\Delta(\check{\Omega})},\ u \in H_{\Delta,\theta}(\check{\Omega}),
\ee
holds for some constant $C>0$ depending only on $\omega$. Indeed, for each $u \in H_{\Delta,\theta}(\check{\Omega})$, where $\theta \in [0,2 \pi)$ is fixed, put 
$f:=\mathcal T_{0,\theta} u \in \sH_\theta(\check{\Gamma})$ and let $v:= \mathcal T_{0,\theta}^{-1} f \in B_\theta$. Evidently $w:=v-u$ satisfies
$$
\left\{ 
\begin{array}{rcll} 
-\Delta w & = & \Delta u , & \mbox{in}\ \check{\Omega},\\ 
w & = & 0,& \mbox{on}\ \check{\Gamma}, \\
w(1,\cdot) - e^{i\theta}  w(0,\cdot) & = & 0, & \mbox{in}\ \omega. \\
\pd_{x_1} w(1,\cdot) - e^{i\theta} \pd_{x_1} w(0,\cdot) & = & 0, & \mbox{in}\ \omega,
\end{array} 
\right.
$$
hence $w=\mathcal A_{0,\theta}^{-1} \Delta u$. Since $\mathcal A_{0,\theta}$ is bounded from below by the Poincar\'e constant $C_\omega>0$ defined in \eqref{p1const}, in virtue of \eqref{qf}, it holds true that $\| w \|_{L^2(\check \Omega)} \leq C_\omega^{-1} \| \Delta u \|_{L^2(\check \Omega)}$. Therefore we have
$$ \| f \|_{\sH_\theta(\check{\Gamma})} = \| v \|_{L^2(\check{\Omega})} \leq \| u \|_{L^2(\check{\Omega})} + \| w \|_{L^2(\check{\Omega})} \leq (1+C_\omega^{-1}) \left( \| u \|_{L^2(\check{\Omega})} + \| \Delta u \|_{L^2(\check{\Omega})} \right), $$
which yields \eqref{esti}.

\subsection{Linking up \eqref{eq1} with \eqref{eq2}}
\label{sec-lin}
Let $Y$ be either $\omega$ or $\pd \omega$, and let $s \in \R_+$ and $k \in \N$.
In light of \eqref{c3} we extend $\mathcal{U}$ to a unitary operator from $\cH^{k,s}(\R \times Y)$ onto the Hilbert space
$\int_{(0,2 \pi)}^\oplus \cH_{\theta}^k(0,1;H^s(Y)) \frac{d \theta}{2 \pi}$.
Next, for $f \in \cH^{-k,-s}(\R \times Y)$, we define $\U f$ by setting
$$ \langle \U f, g \rangle := \langle f, \U^{-1} g \rangle_{\cH^{-k,-s}(\R \times Y), \cH_0^{k,s}(\R \times Y)}\ \mbox{for all}\ g \in \int_{(0,2 \pi)}^\oplus \cH_{\theta}^k(0,1;H_0^{s}(Y)) \frac{d \theta}{2 \pi}. $$
Here $\langle .,.\rangle$ denotes the duality pairing  between $\int_{(0,2 \pi)}^\oplus \cH_{\theta}^{-k}(0,1;H^{-s}(Y)) \frac{d \theta}{2 \pi}$ and $\int_{(0,2 \pi)}^\oplus \cH_{\theta}^{k}(0,1;H_0^{s}(Y)) \frac{d \theta}{2 \pi}$, and we recall that $\cH_0^{k,s}(\R \times Y)=H^k(\R;H_0^s(Y))$, where $H_0^s(Y)$ is the closure of $C_0^{\infty}(Y)$ in the topology of $H^s(Y)$. 
It is clear that the above defined operator $\mathcal U$ is  unitary from $\cH^{-k,-s}(\R \times Y)$ onto the Hilbert space
$\int_{(0,2 \pi)}^\oplus \cH_{\theta}^{-k}(0,1;H^{-s}(Y)) \frac{d \theta}{2 \pi}$. Moreover, for every $f \in C_0^\infty(\R; C^\infty(\overline{Y}))$ and $g \in \int_{(0,2 \pi)}^\oplus \cH_{\theta}^{k}(0,1;H_0^{s}(Y)) \frac{d \theta}{2 \pi}$, we have
$$\langle \U f, g \rangle  = \langle \U f, g  \rangle_{L^2 \left( (0, 2 \pi) \frac{d \theta}{2 \pi};\cH_{\theta}^{k}(0,1;H^{s}(Y) \right) }
= \langle  f , \U^{-1} g \rangle_{L^2(\R \times Y)}, $$
and hence $\langle \U f, g \rangle= \langle f , \U^{-1} g \rangle_{\cH^{-k,-s}(\R \times Y),\cH_0^{k,s}(\R \times Y)}$, which shows that the above definition of $\mathcal U$ is more general than the one of Subsection \ref{sec-fbg}. 

For $j=0,1$, it is apparent that $(\U \mathcal T_j u)_\theta = \mathcal T_{j,\theta} (\U u)_\theta$ for all $u \in C_0^\infty(\overline{\Omega})$ and $\theta \in [0,2\pi)$.
Since the operators $\mathcal T_j : H_\Delta(\Omega) \to \cH^{-2,-(j+\frac{1}{2})}(\Gamma)$ and $\mathcal T_{j,\theta} : H_{\Delta,\theta}(\check{\Omega}) \to \cH_{\theta}^{-2}(0,1; H^{-(j+\frac{1}{2})}(\pd \omega))$ are bounded, we deduce from \eqref{ec1b} and the density of $C_0^\infty(\overline{\Omega})$ in $H_\Delta(\Omega)$
that $\U \mathcal T_j \U^{-1} = \int_{(0,2 \pi)}^{\oplus} \mathcal T_{j,\theta} \frac{d \theta}{2 \pi}$. 
This entails that $\U \sH(\Gamma)=\int_{(0,2 \pi)}^{\oplus} \sH_\theta(\check{\Gamma}) \frac{d \theta}{2 \pi}$. Moreover, arguing as in the proof of \cite[Proposition 3.1]{CKS}, we find for any $f \in \sH(\Gamma)$ that $u$ is the $H_\Delta(\Omega)$-solution to \eqref{eq1} if and only if each function $(\U u)_\theta \in H_{\Delta,\theta}(\check{\Omega})$, for a.e. $\theta \in [0, 2 \pi)$, satisfies \eqref{eq2} where $(\U f)_\theta$ is substituted for $g$. This is provided $V \in L^{\infty}(\Omega;\R)$ verifies \eqref{eq-per} in such a way that \eqref{ec1} holds. Consequently we have:
\bel{ec3}
\mathcal U \Lambda_V \mathcal U^{-1} = \int_{(0,2\pi)}^\oplus \Lambda_{V,\theta} \frac{d \theta}{2 \pi}.
\ee

Let $V_j$, $j=1,2$, be the two potentials introduced in Theorem \ref{thm1}. Then, in light of \eqref{ec3} and (iii) in Propositions \ref{pr-a} and \ref{p9}, we have
\bel{norm-fibre} 
\| \Lambda_{V_1}-\Lambda_{V_2}  \|_{\B\left(\sH(\Gamma),L^2(G)\right)} = \sup_{\theta \in [0,2 \pi)} \| \Lambda_{V_1,\theta}-\Lambda_{V_2,\theta} \|_{\B(\sH_\theta(\check{\Gamma}) ,L^2(\check{G}))},
\ee
from \cite[Section II.2, Proposition 2]{Di}. As the function $\Phi$, given by \eqref{def-Phi}, is non decreasing, then Theorem \ref{thm1} follows readily from \eqref{norm-fibre} and the following statement whose proof is given in Section \ref{sec-thm2}.

\begin{theorem}
\label{thm2} 
Under the conditions of Theorem \ref{thm1}, the following estimate
\bel{thm2a} 
\norm{V_1-V_2}_{H^{-1}(\check{\Omega})}\leq C_\theta \Phi \left(\norm{\Lambda_{{V}_1,\theta}-\Lambda_{{V}_2,\theta}}\right),
\ee
holds for every $\theta \in [0,2\pi)$. Here $C_\theta$ is a positive constant depending only on $T$, $\omega$, $G'$ and possibly on $\theta$, $\Phi$ is defined by \eqref{def-Phi}, and
$\norm{\Lambda_{{V}_1,\theta}-\Lambda_{{V}_2,\theta}}$ denotes the usual norm of $\Lambda_{{V}_1,\theta}-\Lambda_{{V}_2,\theta}$ in $\B(\sH_\theta(\check{\Gamma}),L^2(\check{G}))$. 
\end{theorem}

\begin{remark}
\label{rmk-cst}
Notice from \eqref{norm-fibre}-\eqref{thm2a} and the non decreasing behavior of $\Phi$ on $[0,+\infty)$, that the multiplicative constant $C$ appearing in the right hand side of the stability estimate \eqref{thm1a}, reads
$$ C := \inf_{\theta \in [0,2 \pi)} C_\theta, $$
and hence is independent of $\theta$.
\end{remark}

\section{Complex geometric optics solutions}
\label{sec-cgo} 
In this section we aim for building CGO solutions to the system
\bel{qpe-s}
\left\{
\begin{array}{rcll}
(-\Delta +V) u & = & 0, &  \mbox{in}\ \check{\Omega}, \\
u(1,\cdot) - e^{i\theta} u(0,\cdot) & = & 0, & \mbox{in}\ \omega, \\
\pd_{x_1} u(1,\cdot) - e^{i\theta} \pd_{x_1} u(0,\cdot) & = & 0, & \mbox{in}\ \omega,
\end{array}
\right.
\ee
associated with $V \in L^{\infty}(\check{\Omega};\R)$ and $\theta \in [0, 2 \pi)$.
Namely, given a sufficiently large $\tau>0$, we seek solutions of the form
\bel{p1b}
u(x)=\left( 1+w(x) \right)e^{\zeta \cdot x},\ x \in \check{\Omega},
\ee
to \eqref{qpe-s}, where $\zeta \in i (\theta + 2 \pi \mathbb Z ) \times \mathbb C^2$ is chosen in such a way that $\Delta e^{\zeta \cdot x}=0$ for every $x \in \check{\Omega}$, and $w \in \cH_{0}^{2}(\check{\Omega})$ satisfies the estimate
\bel{p1a}
\| w \|_{H^{s}(\check{\Omega})} \leq C \tau^{s-1},\ s \in [0,2],
\ee
for some positive constant $C$, independent of $\tau$. 

To do that we proceed as follows. We pick $k \in \Z$ and for $\xi \in \mathbb S^1$ we choose $\eta \in \R^2\setminus \{0\}$ such that $\eta \cdot \xi=0$. For $r >0$, we set
\bel{def-ell}
\ell = \ell(k,\eta,r,\theta) :=\left\{ \begin{array}{cl}  \left(\theta + 2\pi ([r]+1) \right) \left( 1 , -2 \pi k \frac{\eta}{\abs{\eta}^2} \right), & \textrm{if}\ k\ \mbox{is even},\\
\left(\theta+ 2 \pi \left( [r]+\frac{3}{2} \right) \right) \left( 1 , -2 \pi k \frac{\eta}{\abs{\eta}^2} \right), & \textrm{if}\ k\ \mbox{is odd},
\end{array}
\right.
\ee
in such a way that $\ell \cdot (2 \pi k,\eta) = \ell' \cdot \xi =0$, where we used the notation $\ell=(\ell_1,\ell') \in \R \times \R^2$. Here $[r]$ stands for the integer part of $r$, that is the unique integer fulfilling $[r] \leq r <[r]+1$. Next, we introduce
\bel{def-tau}
\tau= \tau(k,\eta,r,\theta) := \sqrt{\frac{\abs{\eta}^2}{4}+ \pi^2 k^2+\abs{\ell}^2},
\ee
and notice that
\bel{ll1a} 
2 \pi r < \tau \leq \frac{\abs{(2 \pi k,\eta)}}{2}+4 \pi (r+1) \left( 1 + \frac{\abs{2 \pi k}}{\abs{\eta}} \right).
\ee
Then, putting
\bel{l1a}
\zeta_1 :=\left(i \pi k,-\tau \xi+i\frac{\eta}{2}\right)+i \ell\ \mbox{and}\ \zeta_2:=\left(-i \pi k, \tau \xi - i\frac{\eta}{2}\right)+ i\ell,
\ee
it is easy to check for $j=1,2$, that we have
\bel{l1b}
\zeta_1+\overline{\zeta_2}=i(2 \pi k,\eta)\ \mbox{and}\
\zeta_j \cdot \zeta_j =\im \zeta_j \cdot \re\zeta_j=0,\ \mbox{with}\ \zeta_j \in i (\theta + 2 \pi \Z ) \times \C^2\ \mbox{for}\ j=1,2. 
\ee

Further, we argue as in the derivation of \cite[Proposition 4.1 and Lemma 4.2]{CKS}, and obtain the following:

\begin{lemma}
\label{p1}
Fix $V \in L^\infty(\check{\Omega})$. Let $\theta \in [0,2 \pi)$, $k \in \Z$, and let $\xi \in \mathbb S^1$. Pick $\eta \in \R^2 \setminus\{0\}$ such that $\xi \cdot \eta=0$.  
Then we may find $\tau_0>1$ and $w \in \cH_0^2(\check{\Omega})$ fulfilling \eqref{p1a}, such that for every $\tau \geq \tau_0$, the function $u$ given by \eqref{p1b} with $\zeta=\zeta_1$, where $\zeta_1$ is defined in \eqref{l1a}, is solution to the system \eqref{qpe-s}.
\end{lemma}

We recall from Subsection \ref{sec-lin} that the space $\cH_0^2(\check{\Omega})$ appearing in Lemma \ref{p1} denotes the set 
$\{ u \in H^2(\check{\Omega}),\ u(1,\cdot)-u(0,\cdot)=\pd_{x_1} u(1,\cdot) - \pd_{x_1} u(0,\cdot)=0\ \mbox{in}\ \omega \}$.

\section{A Carleman estimate for the quasi-periodic Laplace operator in $\check{\Omega}$}
\label{sec-ce}

In this section we derive a Carleman estimate for the Laplace operator in $\check{\Omega}$ with quasi-periodic boundary conditions. We proceed by adapting the Carleman inequality of \cite[Lemma 2.1]{BU} to $x_1$-quasi-periodic functions on $\check{\Omega}$.

\begin{proposition}
\label{p2} 
Let $\xi \in \mathbb S^1$ and pick $a$, $b$ in $\R$, with $a<b$, in such a way that we have
$$ \omega \subset \{ x'\in\R^2;\ \xi \cdot x' \in (a, b) \}. $$ 
Put $d:=b-a$. Then for all $\theta \in [0,2 \pi)$ and all $\tau>0$, the estimate
\bea
& & \frac{8\tau^2}{d} \| e^{-\tau \xi \cdot x'} u \|_{L^2(\check{\Omega})}^2 + 2 \tau \| e^{-\tau \xi \cdot x'} (\xi\cdot \nu)^{1 \slash 2} \pd_\nu u \|_{L^2(\check{\Gamma}_\xi^+)}^2
\nonumber \\
& \leq & \| e^{-\tau \xi \cdot x'} \Delta u \|_{L^2(\check{\Omega})}^2 + 2 \tau  \| e^{-\tau \xi \cdot x'} | \xi\cdot \nu|^{1 \slash 2} \pd_\nu u \|_{L^2(\check{\Gamma}_\xi^-)}^2,
\label{p2a}
\eea
holds for every $u \in \cC_\theta^2 \left([0,1] \times \overline{\omega} \right)$ satisfying $u_{\vert \check{\Gamma}}=0$. Here we used the notations $\check{\Gamma}_\xi^\pm:=(0,1) \times \pd \omega_\xi^\pm$.
\end{proposition} 
\begin{proof} 
The operator $e^{-\tau \xi \cdot x'}\Delta  e^{\tau \xi \cdot x'}$ decomposes into the sum  $P_+^1 + P_+' + P_-$, with
$$ P_+^1 :=\pd_{x_1}^2,\ P_+' :=\Delta' + \tau^2\ \mbox{and}\ P_- :=2 \tau \xi \cdot \nabla', $$
where the symbol $\Delta'$ (resp., $\nabla'$) stands for the Laplace (resp., gradient) operator with respect to $x' \in \omega$. Thus we get upon setting $v(x):=e^{-\tau \xi \cdot x'} u(x)$ that 
\beas
\| e^{-\tau \xi \cdot x'} \Delta u \|_{L^2(\check{\Omega})}^2 & = & \| e^{-\tau \xi \cdot x'} \Delta e^{\tau \xi \cdot x'} v \|_{L^2(\check{\Omega})}^2 \\
& = & \| (P^1_+ + P_+' + P_-) v \|_{L^2(\check{\Omega})}^2 \\
& = & \| (P^1_+ + P_+') v \|_{L^2(\check{\Omega})}^2 + \| P_- v \|_{L^2(\check{\Omega})}^2 + 2 \re \langle P_+^1 v , P_- v \rangle_{L^2(\check{\Omega})} + 2 \re \langle P_+' v , P_- v \rangle_{L^2(\check{\Omega})}, 
\eeas
and hence
\bel{p2b}
\| P_- v \|_{L^2(\check{\Omega})}^2 + 2 \re \langle P_+' v , P_- v \rangle_{L^2(\check{\Omega})} \leq \| e^{-\tau \xi \cdot x'} \Delta u \|_{L^2(\check{\Omega})}^2 - 2 \re \langle P_+^1 v , P_- v \rangle_{L^2(\check{\Omega})}. 
\ee
On the other hand, since $v$ is quasi-periodic with respect to $x_1 \in (0,1)$, we find upon integrating by parts that
\bel{p2e} 
\re \langle P_+^1 v , P_- v \rangle_{L^2(\check{\Omega})} = -\tau \int_0^1 \int_\omega \nabla' \cdot (\abs{\pd_{x_1} v(x)}^2 \xi ) dx' dx_1 = -\tau \int_{\check{\Gamma}} \abs{\pd_{x_1} v(x)}^2 \xi \cdot \nu(x)  d\sigma(x)=0.
\ee
Here we used the fact, arising from the homogeneous boundary data $v_{\vert \check{\Gamma}}=0$, that $\pd_{x_1} v$ vanishes on $\check{\Gamma}$. 

Next, as the function $w:=v(x_1,\cdot) \in C^2(\overline{\omega})$ satisfies $w_{\vert \pd \omega}=0$ for a.e. $x_1 \in (0,1)$, we deduce from the following Carleman estimate 
$$
\frac{8\tau^2}{d^2} \| w \|_{L^2(\omega)}^2 + 2\tau \int_{\pd \omega} e^{-2\tau  \xi \cdot x'} \xi\cdot \nu(x') \abs{\pd_\nu e^{\tau  \xi \cdot x'} w(x')}^2 d \sigma(x') 
\leq \| P_- w \|_{L^2(\omega)}^2 + 2 \re \langle P_+' w , P_-w \rangle_{L^2(\omega)},
$$
which is borrowed from \cite[Lemma 2.1]{BU}, that
\beas
& & \frac{8\tau^2}{d^2} \| e^{-\tau \xi \cdot x'} u(x_1,\cdot) \|_{L^2(\omega)}^2 + 2 \tau \int_{\pd \omega} e^{-2\tau \xi \cdot x'} \xi\cdot \nu(x) \abs{\pd_\nu u(x_1,x')}^2d \sigma(x') \\
& \leq & \| P_- v(x_1,\cdot) \|_{L^2(\omega)}^2 + 2 \re \langle P_+' v(x_1,\cdot) , P_- v(x_1,\cdot) \rangle_{L^2(\omega)}.
\eeas
Thus, by integrating both sides of the above inequality with respect to $x_1\in (0,1)$, we get that
\bel{p2c}
\frac{8\tau^2}{d^2} \| e^{-\tau \xi \cdot x'} u \|_{L^2(\check{\Omega})}^2 +2 \tau \int_{\check{\Gamma}} e^{-2\tau\xi\cdot x'} \xi \cdot \nu(x) \abs{\pd_\nu u(x)}^2 d \sigma(x)
\leq \| P_-v \|_{L^2(\check{\Omega})}^2 + 2 \re \langle P_+' v , P_-v \rangle_{L^2(\check{\Omega})}. 
\ee

Finally, putting 
\eqref{p2b}--\eqref{p2c} together and recalling \eqref{xi-isf}, we end up getting \eqref{p2a}.
\end{proof}

Let us now perturb the Laplacian in \eqref{p2a} by the multiplier by $V \in L^{\infty}(\check{\Omega})$. Since
$$  \abs{\Delta u}^2 \leq 2 \left( \abs{(-\Delta +V )u}^2+\norm{V}^2_{L^\infty(\Omega)}\abs{u}^2 \right), $$
we find through elementary computations that
\beas
& & \left( \frac{4 \tau^2}{d} - \| V \|_{L^{\infty}(\check{\Omega})}^2 \right) \| e^{-\tau \xi \cdot x'} u \|_{L^2(\check{\Omega})}^2 + \tau \| e^{-\tau \xi \cdot x'} (\xi\cdot \nu)^{\frac{1}{2}} \pd_\nu u \|_{L^2(\check{\Gamma}_\xi^+)}^2
\nonumber \\
& \leq & \| e^{-\tau \xi \cdot x'} (-\Delta+V) u \|_{L^2(\check{\Omega})}^2 + \tau  \| e^{-\tau \xi \cdot x'} |\xi \cdot \nu|^{\frac{1}{2}} \pd_\nu u \|_{L^2(\check{\Gamma}_\xi^-)}^2.
\eeas
As a consequence we have obtained the:
\begin{follow}
\label{car1}
For $M>0$, let $V \in L^\infty(\Omega)$ satisfy \eqref{eq-per} and $\| V \|_{L^\infty(\Omega)} \leq M$. Then, under the conditions of Proposition \ref{p2}, we have
\beas
& & \frac{2 \tau^2}{d} \| e^{-\tau \xi \cdot x'} u \|_{L^2(\check{\Omega})}^2 + \tau \| e^{-\tau \xi \cdot x'} (\xi\cdot \nu)^{\frac{1}{2}} \pd_\nu u \|_{L^2(\check{\Gamma}_\xi^+)}^2
\nonumber \\
& \leq & \| e^{-\tau \xi \cdot x'} (-\Delta+V) u \|_{L^2(\check{\Omega})}^2 + \tau  \| e^{-\tau \xi \cdot x'} |\xi\cdot \nu|^{\frac{1}{2}} \pd_\nu u \|_{L^2(\check{\Gamma}_\xi^-)}^2,
\eeas
provided $\tau \geq \tau_1:=M (d \slash 2)^{\frac{1}{2}}$.
\end{follow}

Notice from the density of $\{ u \in \cC_\theta^2 \left([0,1] \times \overline{\omega} \right),\ u_{\vert \check{\Gamma}}=0 \}$ in $\{ u \in \cH_\theta^2 (\check{\Omega} ),\ u_{\vert \check{\Gamma}}=0 \}$ that the Carleman estimate of Corollary \ref{car1} remains valid for all $u \in \cH_\theta^2 (\check{\Omega} )$ such that $u_{\vert \check{\Gamma}}=0$.

\section{Proof of Theorem \ref{thm2}}
\label{sec-thm2}
In this section we prove the stability estimate \eqref{thm2a}. To this purpose we set for a.e. $x \in (0,1) \times \R^2$,
$$ V(x):= \left\{ \begin{array}{cl} (V_2-V_1)(x) & \mbox{if}\ x \in \check{\Omega} \\ 0 & \mbox{if}\ x \in (0,1) \times (\R^2 \setminus \omega), \end{array} \right. $$
and start by establishing several technical results that are useful for the proof of \eqref{thm2a}. 

Since we aim for proving the stability estimate \eqref{thm1a} with the aid of \eqref{thm2a}, we recall from Remark \ref{rmk-cst} that we may completely leave aside the question of how the constant $C_\theta$ involved in \eqref{thm2a} depends on $\theta$. Therefore we shall not specify the possible dependence with respect to $\theta$ of the various constants appearing in this section.

\subsection{Preliminary estimate}
\label{sec-pe}
Bearing in mind that $G'$ is a closed neighborhood of $\pd \omega_{\xi_0}^-=\{ x \in \pd \omega;\ \xi_0 \cdot \nu(x) \leq 0 \}$, we pick $\epsilon>0$ so small that
\bel{F'}
\forall \xi \in\mathbb S^1, ( |\xi-\xi_0| \leq \epsilon ) \Rightarrow \left( \{ x \in \pd \omega;\ \xi \cdot \nu(x) \leq \epsilon \} \subset G' \right),
\ee
and we establish the following technical result with the help of the CGO solutions introduced in Section \ref{sec-cgo} and the Carleman estimate derived in Section \ref{sec-ce}.

\begin{lemma}
\label{l100} 
Let $\epsilon$ be as in \eqref{F'}. Then we may find $r_1>0$ such that the estimate
\bel{t3b}
\abs{\int_{\R^2}\int_0^{1} V(x_1,x')e^{-i(2 \pi k x_1+\eta\cdot x')}dx_1dx'}^2\leq C \left(\frac{1}{\tau}+e^{C' \tau}\norm{\Lambda_{{V}_2,\theta}-\Lambda_{{V}_1,\theta}}^2\right),
\ee
holds for all $r \geq r_1$, $\xi \in \{z\in\mathbb S^1;\ \abs{z-\xi_0} \leq \epsilon\}$, $\eta \in \R^2 \setminus \{0\}$ satisfying $\xi \cdot \eta=0$, and $k \in \mathbb Z$.
Here $\tau \geq 1$ is defined by \eqref{def-ell}-\eqref{def-tau}, and the positive constants $C$ and $C'$ depend only on $\omega$, $M_+$, $\epsilon$ and $\xi_0$.
\end{lemma}
\begin{proof} 
We first introduce the sets 
\bel{def-xieps}
\pd \omega_{\xi,\epsilon}^+:=\{ x \in \pd \omega;\  \xi \cdot \nu(x) > \epsilon\}\ \mbox{and}\ \pd \omega_{\xi,\epsilon}^-:=\{ x \in \pd \omega;\ \xi \cdot \nu(x) \leq \epsilon \},
\ee
and we establish the orthogonality identity \eqref{t3a} below, with the aid of the CGO solutions of Lemma \ref{p1}. To this end we choose $r$ sufficiently large, namely
$r \geq r_1:=\max (1+\tau_0,\tau_1 )$, where $\tau_0$ (resp., $\tau_1$) is the constant introduced in Lemma \ref{p1} (resp., Corollary \ref{car1}), in such a way that we have
\bel{ea6} 
\tau=\sqrt{\frac{\abs{\eta}^2}{4}+\pi^2 k^2+\abs{\ell(r,k,\eta,\theta)}^2} \geq  \max(\tau_0,\tau_1),
\ee
by \eqref{def-ell}-\eqref{def-tau}.
Next, for $j=1,2$, we define $\zeta_j$ as in \eqref{l1a} and we denote by 
\bel{ea6b} 
u_j(x)=( 1+w_j(x) ) e^{\zeta_j \cdot x},\ x \in \check{\Omega},
\ee
the $\cH_{\theta}^2(\check{\Omega})$-solution to \eqref{qpe-s} associated with $V=V_j$, which is given by Lemma \ref{p1}. For further reference we recall that $w_j$ satisfies the condition \eqref{p1a} with $s=1$, entailing that
\bel{ea7}
\| w_j \|_{H^1(\check{\Omega})} \leq C,
\ee
for some constant $C=C(\omega,M_+)>0$.
Next, if $v_1$ satisfies the system
\bel{eq3}
\left\{
\begin{array}{rcll}
 (-\Delta + V_1 )v_1 & = & 0, & \mbox{in}\ \check{\Omega},
\\
v_1 & = & \mathcal T_{0,\theta} u_2, & \mbox{on}\ \check{\Gamma}, \\
v_1 (1,\cdot) - e^{i \theta} v_1(0,\cdot) & = & 0, & \mbox{on}\ \omega, \\
\pd_{x_1} v_1 (1,\cdot) - e^{i \theta} \pd_{x_1} v_1(0,\cdot) & = & 0, & \mbox{on}\ \omega,
\end{array}
\right.
\ee
then it is easy to check that the function $u:=v_1-u_2$ is solution to the following BVP:
\bel{eq4}
\left\{ 
\begin{array}{rcll}
(-\Delta + V_1) u & = & V u_2, & \mbox{in}\ \check{\Omega}, \\
u & = & 0, &\mbox{on}\ \check{\Gamma}, \\
u(1,\cdot) - e^{i\theta} u(0,\cdot) & = & 0, & \mbox{on}\ \omega, \\
\pd_{x_1} u(1,\cdot) - e^{i\theta} \pd_{x_1} u(0,\cdot) & = & 0, & \mathrm{on}\ \omega. 
\end{array}
\right.
\ee
Moreover, as $V u_2\in L^2(\check{\Omega})$ and $0$ belongs to the resolvent set of $\mathcal A_{V_1,\theta}$, then $u=\mathcal A_{V_1,\theta}^{-1} V u_2 \in \cH^2_\theta(\check{\Omega})$.
Further, bearing in mind that $(-\Delta+V_1)u_1=0$ in $\check{\Omega}$, from the first line of \eqref{qpe-s} with $V=V_1$, we deduce from \eqref{eq4} and the Green formula that
$$
\int_{\check{\Omega}} V u_2 \overline{u_1} d x =\int_{\check{\Omega}} (-\Delta+V_1)u \overline{u_1}d x
= \int_{\check{\Gamma}} (\pd_\nu u) \overline{u_1} d\sigma(x).
$$
In view of \eqref{def-xieps}, this can be equivalently rewritten as
\bel{t3a}
\int_{\check{\Omega}} V u_2\overline{u_1} d x
=\int_{\check{\Gamma}_{\xi,\epsilon}^+} (\pd_\nu u ) \overline{u_1} d\sigma(x) +\int_{\check{\Gamma}_{\xi,\epsilon}^-} (\pd_\nu u) \overline{u_1} d\sigma(x),
\ee
where $\check{\Gamma}_{\xi,\epsilon}^\pm := (0,1) \times \pd \omega_{\xi,\epsilon}^\pm$.
With reference to \eqref{ea6b}, we deduce from \eqref{ea7} and the continuity of the trace from $H^1(\check{\Omega})$ into $L^2(\check{\Gamma})$, that
\bea
\abs{ \int_{\check{\Gamma}_{\xi,\epsilon}^{\pm}} (\pd_\nu u) \overline{u_1} d\sigma(x)} 
& \leq & \int_{\check{\Gamma}_{\xi,\epsilon}^{\pm}} \abs{(\pd_\nu u ) e^{-\tau \xi \cdot x'} (1+w_1(x))} d\sigma(x') dx_1 \nonumber \\
& \leq & C \| e^{-\tau\xi\cdot x'} \pd_\nu u \|_{L^2(\check{\Gamma}_{\xi,\epsilon}^{\pm})}, \label{t2b}
\eea
where $C$ is another positive constant depending only on $\omega$ and $M_+$. Moreover, we have
$$ 
\epsilon \| e^{-\tau \xi \cdot x'} \pd_\nu u \|_{L^2(\check{\Gamma}_{\xi,\epsilon}^+)}^2 \leq \| e^{-\tau \xi \cdot x'} (\xi \cdot \nu)^{\frac{1}{2}} \pd_\nu u \|_{L^2(\check{\Gamma}_{\xi,\epsilon}^+)}^2
\leq \| e^{-\tau \xi \cdot x'} (\xi \cdot \nu)^{\frac{1}{2}} \pd_\nu u \|_{L^2(\check{\Gamma}_{\xi,\epsilon}^+)}^2, $$
from the very definition of $\pd \omega_{\xi,\epsilon}^+$ and the imbedding $\pd \omega_{\xi,\epsilon}^+ \subset \pd \omega_{\xi}^+$. Therefore, applying the Carleman estimate of Corollary \ref{car1} to the $\cH_\theta^2(\check{\Omega})$-solution $u$ of \eqref{eq4}, which is permitted since $\tau \geq \tau_1$, we get that
\bea
\tau \epsilon \| e^{-\tau \xi \cdot x'} \pd_\nu u \|_{L^2(\check{\Gamma}_{\xi,\epsilon}^+)}^2 
& \leq & C \left( \| e^{-\tau \xi \cdot x'} (-\Delta+V_1) u \|_{L^2(\check{\Omega})}^2 + \tau \| e^{-\tau \xi \cdot x'} |\xi \cdot \nu|^{\frac{1}{2}} \pd_\nu u \|_{L^2(\check{\Gamma}_\xi^-)}^2 \right) \nonumber \\
& \leq & C \left( \| e^{-\tau \xi \cdot x'} V u_2 \|_{L^2(\check{\Omega})}^2 + \tau \| e^{-\tau \xi \cdot x'} |\xi \cdot \nu|^{\frac{1}{2}} \pd_\nu u \|_{L^2(\check{\Gamma}_\xi^-)}^2 \right). \label{ea8}
\eea
Next, as $e^{-\tau \xi \cdot x'} V u_2(x) =  e^{-\tau \xi \cdot x'} V  e^{\zeta_2 \cdot x} (1+w_2(x)) =  e^{-i \left( \pi k x_1 + \frac{\eta \cdot x'}{2} - \ell \cdot x \right)} V (1+w_2(x))$ for a.e. $x \in \check{\Omega}$,
from \eqref{l1a} and \eqref{ea6b}, we have $| e^{-\tau \xi \cdot x'} V u_2(x) | = | V(x) | | 1 + w_2(x) |$ by \eqref{def-ell}, so it follows from \eqref{ea7} that
\bel{ea9}
\| e^{-\tau \xi \cdot x'} V u_2 \|_{L^2(\check{\Omega})} \leq M_+ ( | \omega | + C). 
\ee
Further, bearing in mind  that $\pd \omega_{\xi}^- \subset \pd \omega_{\xi,\epsilon}^-$ and $| \xi \cdot \nu | \leq 1$ on $\pd \omega_{\xi,\epsilon}^-$, we get that
$$ \| e^{-\tau \xi \cdot x'} |\xi \cdot \nu|^{\frac{1}{2}} \pd_\nu u \|_{L^2(\check{\Gamma}_\xi^-)} \leq \| e^{-\tau \xi \cdot x'} \pd_\nu u \|_{L^2(\check{\Gamma}_{\xi,\epsilon}^-)}, $$
which, together with \eqref{ea8}-\eqref{ea9}, yield the estimate
$$
\| e^{-\tau \xi \cdot x'} \pd_\nu u \|_{L^2(\check{\Gamma}_{\xi,\epsilon}^+)}^2 
\leq \frac{C}{\epsilon} \left( \frac{1}{\tau} + \| e^{-\tau \xi \cdot x'} \pd_\nu u \|_{L^2(\check{\Gamma}_{\xi,\epsilon}^-)}^2 \right),
$$
where $C=C(\omega, M_\pm)>0$.
From this and \eqref{t3a}-\eqref{t2b}, it follows that
\bel{t2f}
\abs{\int_{\check{\Omega}} V u_2 \overline{u_1} dx} \leq C \left( \frac{1}{\tau} + \| e^{-\tau \xi \cdot x'} \pd_\nu u \|_{L^2(\check{\Gamma}_{\xi,\epsilon}^-)}^2
\right)^{\frac{1}{2}},
\ee
the positive constant $C$ depending this time on $M_\pm$ and $G'$. 

On the other hand, with reference to \eqref{l1a}-\eqref{l1b} and \eqref{ea6b}, we find through direct calculation that
\bel{ea10}
\int_{\check{\Omega}} V u_2 \overline{u_1} dx= \int_{(0,1) \times \R^2} e^{-i ( 2 \pi k x_1 + \eta \cdot x')} V(x_1,x') dx_1dx'+ \int_{\check{\Omega}} R(x)  dx,
\ee
where
$$ R(x) := e^{-i( 2 \pi k x_1 + \eta \cdot x')} V(x) \left( w_2(x)+ \overline{w_1(x)} +w_2(x) \overline{w_1(x)} \right),\ x=(x_1,x') \in \check{\Omega}.$$ 
Since $\| w_j \|_{L^2(\check{\Omega})}$, for $j=1,2$, is bounded (up to some multiplicative constant) from above by $\tau^{-1}$, according to \eqref{p1a}, we obtain that
$$ \abs{\int_{\check{\Omega}} R(x)  dx} \leq M_+ \left( | \omega |^{\frac{1}{2}} \left( \| w_1\|_{L^2(\check{\Omega})} + \| w_2 \|_{L^2(\check{\Omega})} \right) + 
\| w_1 \|_{L^2(\check{\Omega})}\| w_2 \|_{L^2(\check{\Omega})} \right) \leq C \tau^{-1}, $$
where $C$ is independent of $\tau$. It follows from this and \eqref{t2f}-\eqref{ea10} that
\bea
\abs{\int_{(0,1) \times \R^2} e^{-i( 2 \pi k x_1 + \eta\cdot x')} V(x) dx}^2 
& \leq & C \left( \frac{1}{\tau}+ \| e^{-\tau \xi \cdot x'} \pd_\nu u \|_{L^2(\check{\Gamma}_{\xi,\epsilon}^-)}^2 \right) \nonumber \\
& \leq & C \left( \frac{1}{\tau} + e^{c_\omega \tau} \norm{\pd_\nu u}_{L^2(\check{\Gamma}_{\xi,\epsilon}^-)}^2 \right). \label{t3d}
\eea
where $c_\omega:=\max \{ |x'|; x' \in \overline{\omega} \}$ and $C=C(\omega,M_\pm,G')>0$. Finally, upon recalling that $u=v_1-u_2$, where $v_1$ is solution to \eqref{eq3} and $u_2$ satisfies \eqref{qpe-s} with $V=V_2$, we see that
$$ \pd_\nu u=(\Lambda_{{V}_2,\theta} - \Lambda_{{V}_1,\theta} ) f,\ f=\mathcal T_{0,\theta} u_2. $$
Since $\pd \omega_{\xi,\epsilon}^- \subset G'$, by \eqref{F'}, we have $\norm{\pd _\nu u}_{L^2(\check{\Gamma}_{\xi,\epsilon}^-)} \leq  \norm{\Lambda_{{V}_2,\theta}-\Lambda_{{V}_1,\theta}} \norm{\mathcal T_{0,\theta}u_2}_{\sH_\theta(\check{\Gamma})}$, and hence
$$
\norm{\pd _\nu u}_{L^2(\check{\Gamma}_{\xi,\epsilon}^-)} \leq C \norm{\Lambda_{{V}_2,\theta}-\Lambda_{{V}_1,\theta}} \| u_2 \|_{H_\Delta(\check{\Omega})}
\leq C \tau e^{c_\omega \tau}\norm{\Lambda_{{V}_2,\theta}-\Lambda_{{V}_1,\theta}},
$$
by \eqref{esti} and \eqref{p1a}, the constant $C>0$ depending only on $\omega$, $M_\pm$ and $G'$. This and \eqref{t3d} entail \eqref{t3b}.
\end{proof}

\subsection{Two technical results}
Prior to completing the derivation of Theorem \ref{thm1} we collect two technical results that are needed in the remaining part of proof. 

The first statement, which makes use of the following notation
$$ D_s := \{ x' \in \R^2;\ \abs{x'}<s \},\ s \in (0,+\infty),$$
is borrowed from \cite[Theorem 1]{V} and \cite[Theorem 3]{AE}.

\begin{lemma}
\label{p8} 
For $R>0$, let $f:\ D_{2 R}\subset \R^2 \to \mathbb C$ be real analytic and satisfy the condition 
$$ \exists c>0,\ \exists \varrho \in (0,1],\ \forall \beta \in\mathbb N^2,\ \norm{\pd^\beta f}_{L^\infty(D_{2 R})} \leq \frac{c \abs{\beta}!}{(\varrho R)^{\abs{\beta}}}. $$
Then for any $E \subset D_{\frac{R}{2}}$ with positive Lebesgue measure, we may find two constants $N=N(\varrho,\abs{E},R)>0$ and $\upsilon=\upsilon(\varrho,\abs{E},R) \in (0,1)$, such that we
have:
$$ \norm{f}_{L^\infty(D_R)} \leq N c^{1-\upsilon} \norm{f}_{L^1(E)}^{\upsilon}. $$
\end{lemma}

Let us denote by $\hat{u}$ the Fourier transform with respect to $x' \in \R^2$ of $u$, i.e. 
$$ \hat{u}(\eta) := \frac{1}{2 \pi} \int_{\R^2} u(x') e^{-i \eta \cdot x'} d x',\ \eta \in \R^2. $$
Then the second result is as follows.

\begin{lemma}
\label{lem1} 
There exists $C>0$ such that the estimate 
\bel{lem1a}
\| G \|_{H^{-1}( (0,1) \times\R^2 )}\leq C \left\| \sum_{k \in \Z} (1+\abs{(k,\cdot)}^2)^{-\frac{1}{2}} \widehat{G_k} \right\|_{L^2(\R^2)},
\ee
holds for every $G \in L^2( (0,1)\times\R^2)$, with $G_k(x'):=\int_0^{1} G(x_1,x') e^{-i 2 \pi k x_1} dx_1$, $x'\in\R^2$, $k \in \Z$.
\end{lemma}
\begin{proof}
Let $C$ be any positive constant, satisfying
\bel{lem1b}
\left\| \sum_{k \in \Z} (1+\abs{(k,\cdot)}^2)^{\frac{1}{2}} \widehat{w_k} \right\|_{L^2(\R^2)} \leq C \norm{w}_{H^1((0,1)\times\R^2)},
\ee
for all $w \in H_0^1((0,1)\times\R^2)$. 
Since $G \in L^2((0,1)\times\R^2)$, we have
$$ \langle G , w \rangle_{H^{-1}((0,1) \times \R^2) , H_0^1((0,1) \times \R^2)} = \langle G , w \rangle_{L^2((0,1)\times\R^2)} = \int_{\R^2} \left( \sum_{k \in \Z} \widehat{G_k}(\eta) \overline{\widehat{w_k}(\eta)} \right) d\eta, $$
by Plancherel formula. It follows from this and \eqref{lem1b} that
$$ \abs{ \langle G, w \rangle_{H^{-1}((0,1) \times \R^2),H_0^1((0,1)\times\R^2)}} \leq C 
\left\| \sum_{k \in \Z} (1+\abs{(k,\cdot)}^2)^{-\frac{1}{2}} \widehat{G_k}  \right\|_{L^2(\R^2)} \norm{w}_{H^1((0,1)\times\R^2)},$$
entailing \eqref{lem1a}.
\end{proof}

\subsection{Completion of the proof}
Let us express $\xi_0 \in \mathbb{S}_1$ as $\xi_0=(\cos \alpha_0,\sin \alpha_0)$, where $\alpha_0$ is uniquely defined in $[0,2\pi)$, and for $\epsilon>0$ satisfying \eqref{F'}, pick $\alpha_1 \in (0,\pi)$ such that
\bel{eb1} 
\{(\cos \alpha,\sin \alpha);\ \alpha \in (\alpha_0-\alpha_1,\alpha_0+\alpha_1) \} \subset \{ z \in \mathbb S^1;\ \abs{z-\xi_0} \leq \epsilon\}.
\ee
Next, for $\rho>0$ fixed, and for all $\eta \in \R^2$ and $k \in \mathbb Z$, we introduce
\bel{def-hk}
H_k(\eta):=\widehat{v_k}(\rho\eta) =\frac{1}{2\pi}\int_{\R^2} v_k(x')e^{-i \rho \eta \cdot x'}dx'\ \mbox{with}\
v_k(x') :=\int_0^{1}V(x_1,x')e^{-i 2 \pi k x_1} d x_1.
\ee
Since $v_k$ is supported in $\overline{\omega}$ and $0 \in\omega$ by assumption, the function $H_k$ is analytic in $\R^2$, and we get through elementary computations that
$2 \pi \abs{\pd^\beta H_k(\eta)} \leq \norm{v_k}_{L^1(\omega)} c_\omega^{\abs{\beta}} \rho^{\abs{\beta}} \leq \abs{\omega}^{\frac{1}{2}} \norm{v_k}_{L^2(\omega)} c_\omega^{\abs{\beta}} \rho^{\abs{\beta}}$ for each $\beta\in\mathbb N^2$,
where we recall from Subsection \ref{sec-pe} that $c_\omega=\max \{ |x'|,\ x' \in \omega \}$.
As $\norm{v_k}_{L^2(\omega)} \leq \norm{V}_{L^2(\check{\Omega})} \leq M_+ \abs{\omega}^{\frac{1}{2}}$ from the Plancherel theorem, and $\rho^{\abs{\beta}} \leq \abs{\beta}! e^{\rho}$, this entails that
$$
\abs{\pd^\beta H_k(\eta)}\leq \frac{M_+ \abs{\omega}}{2 \pi} e^{\rho} c_\omega^{\abs{\beta}} \abs{\beta}!,\ \beta\in\mathbb N^2. 
$$
Thus, applying Lemma \ref{p8} with $f=H_k$, $R=c_\omega^{-1}+1$, $c=M_+ \abs{\omega} e^{\rho} \slash (2 \pi)$, $\varrho=(R c_\omega)^{-1}=(1+c_\omega)^{-1} \in (0,1]$, and
\bel{eb2}
E:= \left\{ t (-\sin \alpha,\cos \alpha);\ \alpha \in (\alpha_0-\alpha_1,\alpha_0+\alpha_1),\ t \in [ 0 , \min(1 , R \slash 2) ) \right\},
\ee
we find that
\bel{eb3}
\norm{H_k}_{L^\infty(D_R)}\leq C e^{\rho(1-\upsilon)} \norm{H_k}_{L^\infty(E)}^\upsilon,
\ee
where $C=C(\omega,M_+,G')>0$ and $\upsilon=\upsilon(\omega,M_+,G') \in (0,1)$. Moreover, in view of \eqref{eb1}-\eqref{eb2}, Lemma \ref{l100} tells us that the estimate \eqref{t3b} holds uniformly in $\eta \in E$, i.e.
$\norm{H_k}_{L^\infty(E)}^2 \leq C \left( \tau^{-1} + e^{C'\tau} \norm{\Lambda_{{V}_2,\theta}-\Lambda_{{V}_1,\theta}}^2 \right)$,
provided $r \geq r_1$. It follows from this and \eqref{eb3} that
\bel{t3c}
\abs{\widehat{v_k}(\eta)}^2 \leq C e^{2\rho(1-\upsilon)} \left( \frac{1}{\tau} + e^{C'\tau} \norm{\Lambda_{{V}_2,\theta}-\Lambda_{{V}_1,\theta}}^2 \right)^\upsilon,\ \abs{(k,\eta)}<\rho,\ r \geq r_1.
\ee
The next step is to apply Lemma \ref{lem1} with $G=V$: With reference to \eqref{def-hk}, we obtain that
\bea
\| V \|_{H^{-1}((0,1)\times\R^2)}^2 & \leq & C \int_{\R^2} \sum_{k =-\infty}^{+\infty} (1+\abs{(k,\eta)}^2)^{-1}| \widehat{v_k}(\eta)|^2  d \eta \nonumber \\
& \leq & C \int_{\R^3} (1+\abs{(k,\eta)}^2)^{-1}| \widehat{v_k}(\eta) |^2  d \mu(k) d \eta, \label{t3e}
\eea
where $\mu:=\sum_{n \in \Z} \delta_n$. Putting $B_\rho:=\{(k,\eta) \in \R^3,\ | (k,\eta) | \leq \rho \}$  we now examine the two integrals
$\int_{B_\rho} (1+\abs{(k,\eta)}^2)^{-1}| \widehat{v_k}(\eta) |^2   d \mu(k) d \eta$ and $\int_{\R^3 \setminus B_\rho} (1+\abs{(k,\eta)}^2)^{-1}| \widehat{v_k}(\eta) |^2  d \mu(k) d \eta$ separately. The last one, is easily treated, as we have
$$
\int_{\R^3 \setminus B_{\rho}} (1+\abs{(k,\eta)}^2)^{-1}| \widehat{v_k}(\eta) |^2  d \mu(k) d \eta 
\leq \frac{1}{\rho^2} \int_{\R^3 \setminus B_{\rho}} | \widehat{v_k}(\eta) |^2  d \mu(k) d \eta 
\leq \frac{1}{\rho^2} \int_{\R^3} | \widehat{v_k}(\eta) |^2  d \mu(k) d \eta, 
$$
and hence
\bel{t3f}
\int_{\R^3 \setminus B_{\rho}} (1+\abs{(k,\eta)}^2)^{-1}| \widehat{v_k}(\eta) |^2  d \mu(k)d \eta  \leq \frac{1}{\rho^2} \int_{(0,1) \times \R^2} | V(x_1,x') |^2 d x_1d  x' \\ 
\leq \frac{\abs{\omega} M_+^2}{\rho^2},
\ee
by Parseval-Plancherel theorem.
We turn now to studying the first integral. To this end we notice upon setting $\cC_\rho:=\R \times D_{1 \slash \rho}$, where we recall that $D_{1 \slash \rho}=\{ \eta \in \R^2;\ | \eta | < 1 \slash \rho \}$, that
\bea
& & \int_{B_{\rho} \cap \cC_\rho} (1+\abs{(k,\eta)}^2)^{-1} | \widehat{v_k}(\eta) |^2 d\eta d \mu(k) \nonumber \\
& \leq & \int_{\cC_\rho} | \widehat{v_k}(\eta) |^2  d \mu(k) d \eta \leq \int_{D_{1 \slash \rho} } \left( \sum_{k \in \Z} | \widehat{v_k}(\eta) |^2 \right) d \eta
\leq \frac{\pi}{\rho^2} \norm{\sum_{k \in \Z} | \widehat{v_k} |^2}_{L^{\infty}(D_{1 \slash \rho})}. \label{eb4}
\eea
Moreover, since each $v_k$, for $k \in \Z$, is supported in $\omega$ by \eqref{def-hk}, we have
$$ | \widehat{v_k}(\eta) |^2 \leq \frac{1}{4 \pi^2} \| v_k \|_{L^1(\omega)}^2 \leq \frac{|\omega|}{4 \pi^2} \| v_k \|_{L^2(\omega)}^2,\ \eta \in \R^2, $$
which entails
$$ \sum_{k \in \Z} | \widehat{v_k}(\eta) |^2 \leq \frac{| \omega |}{4 \pi^2} \left( \sum_{k \in \Z} \| v_k \|_{L^2(\omega)}^2 \right) \leq \frac{| \omega |}{4 \pi^2} \| V \|_{L^2(\check \Omega)}^2 \leq \frac{M_+^2 | \omega |^2}{4 \pi^2},\ \eta \in \R^2, $$
upon applying the Parseval formula and the dominated convergence theorem.
Putting this together with \eqref{eb4}, we obtain that
\bel{t3g}
\int_{B_{\rho} \cap \cC_\rho} (1+\abs{(k,\eta)}^2)^{-1} | \widehat{v_k}(\eta) |^2 d\eta d \mu(k) \leq \frac{M_+^2 | \omega |^2}{4 \pi \rho^2}.
\ee
Further, if $(k,\eta) \in B_\rho \cap (\R^3 \setminus \cC_\rho)$, then there exists a positive constant $C$ such that $\tau(k,\eta,\theta,r) \in (r, C \rho^2 r]$ for all $\rho \geq 1$ and $r \geq r_1$, according to \eqref{ll1a}. As a consequence we have
$$
\abs{\widehat{v_k}(\eta)}^2 \leq Ce^{2\rho(1-\upsilon)} \left(\frac{1}{r} + e^{C \rho^2 r} \norm{\Lambda_{{V}_2,\theta}-\Lambda_{{V}_1,\theta}}^2\right)^\upsilon,\ r \geq r_1,\ (k,\eta) \in B_\rho \cap (\R^3 \setminus \cC_\rho),$$
by \eqref{t3c}, and hence
\bel{t3h}
\int_{ B_{\rho} \cap (\R^3 \setminus \cC_\rho)} (1+\abs{(k,\eta)}^2)^{-1} | \widehat{v_k}(\eta) |^2   d \mu(k) d \eta
\leq C \rho^3 e^{2\rho(1-\upsilon)} \left(\frac{1}{r}+e^{C \rho^2 r} \norm{\Lambda_{{V}_2,\theta}-\Lambda_{{V}_1,\theta}}^2\right)^\upsilon,\ r \geq r_1,
\ee
upon eventually substituting $C$ for some suitable algebraic expression of $C$.

Therefore, putting \eqref{t3e}-\eqref{t3f} and \eqref{t3g}-\eqref{t3h} together, we find for all $\rho \geq 1$ and $r \geq r_1$ that
\bea
\| V \|_{H^{-1}((0,1)\times\R^2)}^{\frac{2}{\upsilon}} &\leq & C \left( \frac{1}{\rho^2}+ \rho^3 e^{2\rho(1-\upsilon)} \left(\frac{1}{r} + e^{C \rho^2 r}  \norm{\Lambda_{{V}_2,\theta}-\Lambda_{{V}_1,\theta}}^2 \right)^\upsilon \right)^{\frac{1}{\upsilon}} \nonumber \\
&\leq & C \left( \rho^{-\frac{2}{\upsilon}} +\rho^{\frac{3}{\upsilon}} e^{2 \rho \left( \frac{1}{\upsilon} -1 \right)} r^{-1} + \rho^{\frac{3}{\upsilon}} e^{2 \rho \left( \frac{1}{\upsilon} -1 \right)} e^{C\rho^2 r} \norm{\Lambda_{{V}_2,\theta}-\Lambda_{{V}_1,\theta}}^2\right). \label{t3j}
\eea
Here we substituted $C=C(\omega,M_+,F',\upsilon)>0$ for $2^{\frac{1}{\upsilon}} C$ in the last line.
Thus, taking 
$r:=\rho^{\frac{5}{\upsilon}} e^{2\rho \left( \frac{1}{\upsilon}-1 \right)}$ in \eqref{t3j}, with $\rho \geq\rho_1$, in such a way that $r \geq r_1$ and 
$\rho^{\frac{3}{\upsilon}}  e^{2\rho \left( \frac{1}{\upsilon}-1 \right)} r^{-1} = \rho^{-\frac{2}{\upsilon}}$, we obtain that
\bel{t3k}
\| V \|_{H^{-1}((0,1)\times\R^2)} \leq 
C\left( \rho^{-\frac{2}{\upsilon}}+\rho^{\frac{3}{\upsilon}} e^{2\rho \left( \frac{1}{\upsilon}-1 \right)} \exp \left(C{\rho}^{\left( 2 + \frac{5}{\upsilon} \right)} 
e^{2\rho \left( \frac{1}{\upsilon}-1 \right)} \right) \norm{\Lambda_{{V}_2,\theta}-\Lambda_{{V}_1,\theta}}^2 \right)^{\frac{\upsilon}{2}}.
\ee
Let $C'>0$ be so large that $\rho^{\frac{3}{\upsilon}} e^{2\rho \left( \frac{1}{\upsilon}-1 \right)} \exp \left( C{\rho}^{\left( 2 + \frac{5}{\upsilon} \right)} e^{2\rho \left( \frac{1}{\upsilon}-1 \right)} \right) \leq \exp \left( e^{C' \rho} \right)$ for every $\rho \geq \rho_1$. Notice that $C'$ depends only on $C$ and $\upsilon$, hence on $\omega$, $M_+$ and $G'$. Then \eqref{t3k} entails that
\bel{t3l}
\| V \|_{H^{-1}((0,1)\times\R^2)} \leq 
C\left( \rho^{-\frac{2}{\upsilon}} + \exp \left( e^{C' \rho} \right) \norm{\Lambda_{{V}_2,\theta}-\Lambda_{{V}_1,\theta}}^2 \right)^{\frac{\upsilon}{2}},\ \rho \geq \rho_1.
\ee
Set $\gamma^*:=\exp \left(-e^{C' \rho_1} \right) \in (0,1)$. If $\gamma:=\norm{\Lambda_{{V}_2,\theta}-\Lambda_{{V}_1,\theta}} \in (0,\gamma^*]$, we find upon taking
$\rho=\frac{1}{C'} \ln(\abs{\ln \gamma}) \geq \rho_1$  in \eqref{t3l}, that
\bel{eb5}
\| V \|_{H^{-1}(\check{\Omega})} \leq \| V \|_{H^{-1}((0,1)\times\R^2)} \leq C \left( {C'}^{\frac{2}{\upsilon}} + \gamma (\ln \abs{\ln \gamma} )^{\frac{2}{\upsilon}} \right)^\frac{\upsilon}{2} (\ln \abs{\ln \gamma})^{-1} .
\ee
Further, since 
$\sup_{\gamma \in (0,\gamma_*]} \left( {C'}^{\frac{2}{\upsilon}} + \gamma (\ln \abs{\ln \gamma} \right)^{\frac{2}{\upsilon}}$ is just another positive constant depending only on $\omega$, $M$ and $G'$, we end up getting from \eqref{eb5} that
\bel{t3n}
\| V \|_{H^{-1}(\check{\Omega})}\leq C (\ln \abs{\ln \gamma} )^{-1},\ \gamma \in (0,\gamma^*].
\ee
Finally, as $\| V \|_{H^{-1}(\check{\Omega})}\leq C \| V \|_{L^\infty(\check{\Omega})}\leq (C M_+ \slash \gamma^*) \gamma$ for $\gamma > \gamma^*$, where $C>0$ is a constant depending only on $\omega$, we deduce \eqref{thm2a} from this and \eqref{t3n}.

\bigskip

\vspace*{.5cm}
\noindent {\sc Mourad Choulli}, Universit\'e de Lorraine, Institut Elie Cartan de Lorraine, CNRS, UMR 7502, Boulevard des Aiguillettes, BP 70239, 54506 Vandoeuvre les Nancy cedex - Ile du Saulcy, 57045 Metz cedex 01, France.\\
E-mail: {\tt  mourad.choulli@univ-lorraine.fr}. \vspace*{.1cm} \\

\noindent {\sc Yavar Kian}, Aix-Marseille Universit\'e, CNRS, CPT UMR 7332, 13288 Marseille, and Universit\'e de Toulon, CNRS, CPT UMR 7332, 83957 La Garde, France.\\
E-mail: {\tt yavar.kian@univ-amu.fr}. \vspace*{.1cm} \\

\noindent {\sc Eric Soccorsi}, Aix-Marseille Universit\'e, CNRS, CPT UMR 7332, 13288 Marseille, and Universit\'e de Toulon, CNRS, CPT UMR 7332, 83957 La Garde, France.\\
E-mail: {\tt eric.soccorsi@univ-amu.fr}.

\end{document}